\documentclass[10pt, reqno]{amsart}
\usepackage{amsmath, amssymb, amsthm, fancyhdr, verbatim, graphicx}
\usepackage{MnSymbol}
\usepackage{enumerate}
\usepackage[all]{xy}
\usepackage[usenames,dvipsnames]{xcolor}
\usepackage{mathrsfs}
\usepackage{tikz-cd}
\usepackage[T1]{fontenc} 
\usepackage{framed, hyperref}
\usepackage[OT2, T1]{fontenc}
\usepackage[titletoc]{appendix}

\numberwithin{equation}{subsection}

\DeclareSymbolFont{cyrletters}{OT2}{wncyr}{m}{n}
\DeclareMathSymbol{\Sha}{\mathalpha}{cyrletters}{"58}

\usepackage{color}

\newcommand{\F}{\mathbf{F}}

\newcommand{\CC}{\mathbf{C}}

\newcommand{\tr}[0]{\operatorname{tr}}
\newcommand{\wt}[1]{\widetilde{#1}}

\newcommand{\Z}{\mathbf{Z}}
\newcommand{\leg}[2]{\left(\frac{#1}{#2}\right)}
\newcommand{\mf}[1]{\mathfrak{#1}}

\newcommand{\Gal}{\operatorname{Gal}}

\newcommand{\cal}[1]{\mathcal{#1}}
\newcommand{\R}{\mathbb{R}}

\newcommand{\ol}[1]{\overline{#1}}
\newcommand{\wh}[1]{\widehat{#1}}

\newcommand{\Cal}[1]{\mathcal{#1}}

\newcommand{\co}{\colon}
\newcommand{\mrm}[1]{\mathrm{#1}}
\newcommand{\msf}[1]{\mathsf{#1}}
\newcommand{\floor}[1]{\lfloor #1 \rfloor}

\DeclareMathOperator{\GL}{GL}
\DeclareMathOperator{\SL}{SL}
\DeclareMathOperator{\Frob}{Frob}
\DeclareMathOperator{\ab}{ab}
\DeclareMathOperator{\coker}{coker}

\DeclareMathOperator{\N}{\mathbb{N}}
\DeclareMathOperator{\Tr}{Tr}

\DeclareMathOperator{\SO}{SO}
\DeclareMathOperator{\Hom}{Hom}
\DeclareMathOperator{\Ind}{Ind}

\DeclareMathOperator{\Aut}{Aut}

\DeclareMathOperator{\Nm}{Nm}

\DeclareMathOperator{\Art}{Art}

\DeclareMathOperator{\Ver}{Ver}

\DeclareMathOperator{\Res}{Res}

\DeclareMathOperator{\Stab}{Stab}

\DeclareMathOperator{\Id}{Id}

\DeclareMathOperator{\Sym}{Sym}

\DeclareMathOperator{\ind}{ind}
\DeclareMathOperator{\SU}{SU}
\DeclareMathOperator{\disc}{disc}

\newtheorem{thm}{Theorem}[section]
\newtheorem{lemma}[thm]{Lemma}
\newtheorem{prop}[thm]{Proposition}
\newtheorem{cor}[thm]{Corollary}

\theoremstyle{remark}
\newtheorem{remark}[thm]{Remark} 
\newtheorem{defn}[thm]{Definition}

\makeatletter
\def\th@remark{%
  \thm@headfont{\bfseries}%
  \normalfont 
  \thm@preskip \thm@preskip 
  \thm@postskip\thm@preskip
}
\def\imod#1{\allowbreak\mkern5mu({\operator@font mod}\,\,#1)}
\makeatother

\title{Epipelagic Langlands parameters and L-packets for unitary groups}

\author{Tony Feng, Niccolo Ronchetti, Cheng-Chiang Tsai}

\begin{document}

\begin{abstract}
Reeder and Yu have recently given a new construction of a class of supercuspidal representations called \emph{epipelagic representations} \cite{RY14}. We explicitly calculate the Local Langlands Correspondence for certain families of epipelagic representations of unitary groups, following the general construction of Kaletha \cite{Kal15}. The interesting feature of our computation is that we find simplifications within $L$-packets of the two novel invariants introduced in \cite{Kal15}, the \emph{toral invariant} and the \emph{admissible L-embedding}.
\end{abstract}

\maketitle

\tableofcontents

\section{Introduction}

The purpose of this paper is to explicitly compute the local Langlands correspondence for certain \emph{epipelagic representations} of (special) unitary groups. The epipelagic representations are a class of supercuspidal representations with minimal positive depth introduced and studied by Reeder and Yu in \cite{RY14}, where they discovered a systematic construction of epipelagic representations.

We should clarify what we mean by ``the'' local Langlands correspondence. Kaletha has given an explicit construction of $L$-packets of epipelagic representations in \cite{Kal15}. His construction is compatible with a plethora of expected properties of a Langlands correspondence, and this is what we take to be the ``local Langlands correspondence''. 

So, at a high level this paper is simply an explication of Kaletha's construction in the case of unitary groups. However, carrying out Kaletha's recipe is not totally straightforward even in these special cases. It involves a number of intricate calculations, and the main contribution of the present work is to simplify and interpret the output of these calculations in a manner that clarifies the ultimate shape of the L-packets. In doing so, we discover an interesting structural feature, which is however a little technical to state and will be explained over the course of the introduction. 

The motivation for the computation here comes from a desire to understand the relationship between Kaletha's construction of the Langlands correspondence for epipelagic representations and an earlier suggestion by Reeder-Yu \cite[\S 7]{RY14}. To explain this, we need to delve a little more into the details and history of epipelagic representations. \\
	
\subsection{The work of Reeder-Yu and of Kaletha}
Given a group $G$ over a local field $F$, one has the Bruhat-Tits building $\Cal{B}_G(F)$. For each point $x \in \Cal{B}_G(F)$ there is a (decreasing) \emph{Moy-Prasad filtration} $\{G_{x, r} \co  r \in \R_{\ge 0} \}$ whose jumps are indexed by the non-negative multiples of $\frac{1}{m}$, where $m \in \N$ depends on $G$ and $x$. From its subquotients one extracts an algebraic group $\wt{\msf{G}}_x$ (with the property that $\wt{\msf{G}}_x(k_F) = G_{x,0} / G_{x,0+}$) acting on a representation $\msf{V}_x = G_{x,r(x)} / G_{x, r(x)+}$ over the residue field $k_F$ of $F$. 

In \cite[\S 2]{RY14}, an (irreducible) epipelagic representation is built out of the following data: 
\begin{enumerate}
\item a functional $\lambda$ on $\msf{V}_x$, which is stable (in the sense of GIT) for the action of $\wt{\msf{G}}_x$, 
\item a character $\psi$ of a certain finite group stabilizing $\lambda$. 
\end{enumerate}
Roughly speaking, a finite direct sum of irreducible supercuspidal representations are built out of $\lambda$ by an induction process, and the character $\psi$ is used to pick out an irreducible summand. We will denote by $\rho_{\lambda,\psi}$ the epipelagic representation attached to the data $(\lambda, \psi)$.

When Reeder and Yu introduced the notion of epipelagic representations, they considered the problem of attaching Langlands parameters. In the case of an \emph{absolutely simple, simply connected} group they gave a template for attaching a Langlands parameter to a \emph{stable functional} $\lambda$ \cite[\S 7.2-7.3]{RY14}, in the form of an algorithm with a couple steps where certain choices are not uniquely specified. The algorithm goes through the Vinberg-Levy theory of graded Lie algebras. 

Shortly afterwards, Kaletha gave a completely different construction of the Langlands correspondence for epipelagic representations, using a strategy inspired by \cite{DR09}. The rough idea is to try to factor a Langlands parameter through the L-group of a torus, obtaining a character of a subtorus of $G$ by the local Langlands correspondence for tori, and then try to induce up to a representation of $G$. Notably, Kaletha observed two novel subtleties in this case: (1) the character needs to be modified by a \emph{toral invariant}, which plays a similar role to the rectifying character of Bushnell-Henniart, and (2) one needs to make a specific choice of \emph{admissible embedding} of the L-group of the torus. For more details, see the introduction of \cite{Kal15}. 

On \cite[p.7]{Kal15} Kaletha raises the question of comparing the Reeder-Yu template, which was suggested only for \emph{simply connected, absolutely simple} groups to his own construction. In a private communication, Kaletha emphasized to us that the two new subtleties, the toral invariant and the admissible embedding, seem invisible in the Reeder-Yu algorithm, suggesting that either the Vinberg-Levy theory used by Reeder-Yu might encode them implicitly, or that there might be some sort of cancellation in the special setting of simply connected absolutely simple groups that obviates the need to consider them. (Both the toral invariant and admissible embedding are parametrized by choices of signs corresponding to the roots of $G$, so it really makes sense to speak of them ``cancelling out''.) 

Such a cancellation, which is extremely nonobvious from the definitions, should have interesting arithmetic significance. It was with this in mind that we set out to understand some examples for non-split, absolutely simple and simply connected groups. Based on earlier (unpublished) work of the third-named author, CCT, we suspected that $\SU_n$ would be a good place to start. 

We should note that Kaletha's had earlier shown \cite{K13} that in the setting of \emph{simple wild representations}, which are some particularly simple instances of epipelagic representations, the toral invariant and $L$-embedding turn out to be negligible. Kaletha explained to us that this was his main motivation for suspecting a cancellation in the setting of Reeder and Yu. 

\begin{remark}\label{size L-packets}
The work of \cite{Kal15} confirms the predictions of \cite{RY14} for the L-packet, at least in their most conservative form. However, the results \cite{RY14} are suggestive of more ambitious predictions. 

In particular, it was striking to the authors that Reeder and Yu showed that the size of the L-packet corresponding to an epipelagic representation for an absolutely simple, semisimple, simply connected group built from $(\lambda, \psi)$ agrees exactly with the number of permissible $\psi$, according to standard conjectures on the Langlands correspondence \cite[p.466-467]{RY14}. Since the Reeder-Yu template makes reference only to $\lambda$, the numerics seem to suggest that the L-packets might have the shape $\{\rho_{\lambda, \psi} \text{ for all possible $\psi$}\}$. Moreover, this was known to be true by earlier work of Kaletha \cite{K13} for the \emph{simple supercuspidals}, a class of representations constructed by Gross-Reeder \cite{GR10} which are special cases of, and precursors to, epipelagic representations. We emphasize though that this was an initial guess of the first two named authors, TF and NR, which was never stated by Reeder-Yu, and indeed it turns out to be incorrect in general. 
\end{remark}

\subsection{Outline of results}

We may now give a rough description of our findings. A more precise statement of results appears in \S \ref{ssec: precise summary}; however due to the level technicality needed there, we thought it would be helpful to begin with a more informal summary. 

We study the epipelagic representations coming from vertices in the Bruhat-Tits building for the (special) unitary groups $G=\SU_n$ and $G=\mrm{U}_n$ associated to a ramified quadratic extension. Using classical results in invariant theory, we identify stable functionals on $\msf{V}_x$ for $x$ a vertex in $\Cal{B}_G(F)$. (Although our original interest was in $\mrm{SU}_n$, we later realized that all its epipelagic representations under consideration were restrictions from $\mrm{U}_n$, and this was useful for computing the Langlands parameter.) The basics of epipelagic representations are recalled in \S \ref{section epipelagic representations}, and our particular representations of interest $\rho_{\lambda, \psi}$ are constructed in \S \ref{section construction epip reps}.

After briefly recalling Kaletha's construction of epipelagic L-packets from epipelagic Langlands parameters in \S \ref{section Langlands parameters}, we run it in reverse to describe the Langlands parameter attached to $\rho_{\lambda, \psi}$ in \S \ref{section calculation}. One has to compute the (character associated to the) toral invariant $\epsilon_{\lambda}$, and the admissible embedding ${}^L j$, which involves lengthy calculations of $\chi$-data, etc. 

Although \cite{Kal15} completely prescribes the choices needed to determine $\epsilon_{\lambda}$ and ${}^L j$, the raw answer is too complicated to see what is going on. The novel aspect of our calculations here is in distilling this raw answer to a form that elucidates the structure of the L-packets. We will state the result for $G= \mrm{U}_n$; a similar but slightly more complicated description holds for $\SU_n$.

\begin{remark}
We warn at the outset that we are working with a crude notion of $L$-packets: our ``$L$-packets'' consist of those representations of the single group $G$ which have the same Langlands parameter. In other words, we are ignoring representations of inner forms. 
\end{remark}

The epipelagic representations built from a stable functional $\lambda$ coincide with those built from another stable functional $\lambda'$ conjugate to $\lambda$ under $\wt{\msf{G}}_x(k_F)$; thus epipelagic representations are really indexed by (rational) conjugacy classes in $\wt{\msf{G}}_x(k_F)$. We find that if $\rho_{\lambda, \psi}$ and $\rho_{\lambda', \psi'}$ appear in the same $L$-packet, then $\lambda$ and $\lambda'$ are conjugate under the $\wt{\msf{G}}_x(\ol{k}_F)$-action on $\msf{V}_x(\ol{k}_F)$. To emphasize the difference between this and the previous notion of conjugacy, in this case we say that $\lambda$ and $\lambda'$ are \emph{stably conjugate}. The question of which pairs $(\lambda, \psi)$ appear comes down to a recipe involving the toral invariant $\epsilon_{\lambda}$ and ${}^L j$. Now, it is possible to parametrize the space of choices for $\epsilon_{\lambda}$ and ${}^L j$ by the same group $C_{\lambda}^{\vee}$.  Furthermore, there is a canonical identification $C_{\lambda} \cong C_{\lambda'}$ when $\lambda$ and $\lambda'$ are stably conjugate, which we use to view these parameters in the same group. Motivated by the possibility of an interesting ``cancellation'' between the the two, we consider the ``difference'' between $\epsilon_{\lambda}$ and ${}^L j$ in $C_{\lambda}^{\vee}$. Curiously, the result depends qualitatively on the whether the rank of our unitary group is even or odd. 

For $\mrm{U}_n$ with $n$ \emph{even}, we find that the toral invariant $\epsilon_{\lambda}$ depends only on the \emph{stable} conjugacy class of $\lambda$, i.e. it is constant under $\wt{\msf{G}}(\ol{k}_F)$-orbits (whereas a priori it is a function of $\wt{\msf{G}}(k_F)$-orbits). This implies that the toral invariant is constant among representations in the same $L$-packet. This constancy is fairly non-obvious from the raw calculation, and is established in Corollary \ref{cor: toral invariant 2}. The data of the $L$-embedding ${}^L j$ is still quite complicated, and in this case we don't seem to have a ``nice'' description of the $L$-packets.

 For $\mrm{U}_n$ with $n$ \emph{odd}, the picture of the $L$-packets is somewhat more satisfactory. The rational orbits within the stable orbits can also be naturally parametrized by $C_{\lambda}^{\vee}$, and we find that the relative position between $``\epsilon_{\lambda}- {}^L j"$ for two different pairs $\rho_{\lambda, \psi}$ and $\rho_{\lambda', \psi'}$ appearing in the same L-packet coincides exactly with the relative position of the rational orbits of $\lambda$ and $\lambda'$ within their stable orbit. The precise statement appears in Proposition \ref{kostant identity}. In fact the characters $\psi$ are also parametrized by the same group $C_{\lambda}^{\vee}$, and we find:

\begin{thm}\label{main}
For the representations of the form $\rho_{\lambda, \psi}$ of $\mrm{U}_n$, with $n$ odd, constructed in \S \ref{section construction epip reps}, two such representations $\rho_{\lambda, \psi}$ and $\rho_{\lambda', \psi'}$ lie in the same L-packet if and only if $\lambda + \psi =  \lambda' + \psi'$ as elements of $C_{\lambda}^{\vee}$. 
\end{thm}

The more precise statements, which however require some more notation and explanation, appear in Corollary \ref{unitary epi param} and Corollary \ref{special epi param}.

Our computation that the difference between $\epsilon_{\lambda}$ and ${}^L j$ is the relative position of conjugacy classes confirms a suspicion, stemming from the work of \cite{RY14}, that the toral invariant and L-embedding might ``cancel out'' to something simpler for certain types of groups. However, it also rebuffs our initial expectations, also stemming from the results of \cite[Remark \ref{size L-packets}]{RY14}, that the L-packets would consist of $\rho_{\lambda, \psi}$ with fixed $\lambda$. In fact we find the complete opposite: 

\begin{cor}
For $\mrm{U}_n$, with $n$ odd, the L-packet of $\rho_{\lambda, \psi}$ consists of $\{\rho_{\lambda_i, \psi_i}\}$ where $\{\lambda_i\}$ represents an enumeration of the rational orbits within the stable orbit of $\lambda$, and $\{\psi_i\}$ is an enumeration of the characters of $C_{\lambda}$. 
\end{cor}

\begin{remark}
The later paper \cite{Kal17} offers a different construction of the Langlands correspondence for a much more general class of representations, and also features a type of ``cancellation'' in which the particular choice of $L$-embedding in \cite{Kal15} becomes irrelevant. It is not clear to  what extent, if any, the two types of cancellation are related, although they seem to be at least philosophically connected. 
\end{remark}

The restrictions $\rho_{\lambda, \psi}|_{\mrm{SU}_n}$ are already irreducible (Corollary \ref{SU_n restriction}), so their Langlands parameters can be deduced from the ones for $\rho_{\lambda, \psi}$. This leads to: 

\begin{cor}
For $\mrm{SU}_n$, with $n$ odd, the L-packet of $\rho_{\lambda, \psi}|_{\mrm{SU}_n}$ consists of $\{\rho_{\lambda_i, \psi_i}|_{\mrm{SU}_n}\}$ where $\{\lambda_i\}$ represents an enumeration of the rational orbits within the stable orbit of $\lambda$, and $\{\psi_i\}$ is an enumeration of the characters of $C_{\lambda} \cap \mrm{SU}_n$ with multiplicity 2. 
\end{cor}

The outcome of the computation in the case where $n$ is even is not sufficiently clear for us to deduce these sorts of qualitative statements in that case. 

\subsection{Precise statement of results}\label{ssec: precise summary} We now give a precise summary of the results. Unfortunately, they require a heavy amount of notation to state. 

Let $G = \mathrm U_n$ be the unitary group in $n$ variables associated to a tamely ramified quadratic extension $E/F$. The hermitian form defining $G$ puts on $E^n$ the structure of an inner product space, which we denote by $(V , \langle \cdot, \cdot \rangle)$.

\subsubsection{Setup to produce an epipelagic representation}
Let $x$ be a hyperspecial point in the building $\Cal{B}_G(E)$ of $G$ over $E$, such that $x$ is not hyperspecial in the building $\Cal{B}_G(F)$ over $F$.

We are interested in the following successive quotients of the Moy-Prasad filtration for the point $x$:
\begin{itemize}
\item $G_{x} / G_{x, 1/2}$, which can be identified with the group of $k_F$-points of an algebraic group $\wt{\msf{G}}_x$;
\item $G_{x, 1/2} / G_{x,1}$.
\end{itemize}
The conjugation action of $G_{x} / G_{x, 1/2}$ on $G_{x, 1/2} / G_{x,1}$ can be promoted to an algebraic representation $\msf{V}_x$ of $\wt{\msf{G}}_x$ over $k_F$.

Let $\lambda: \msf{V}_x \rightarrow k_F$ be a functional which is GIT-stable for $\wt{\msf{G}}_x$, i.e. so $\lambda$ is an element of the contragredient representation of $\msf{V}_x$ whose stabilizer under the $\wt{\msf{G}}_x$-action is a finite algebraic group, and whose $\wt{\msf{G}}_x$-orbit is closed.

Let $\chi_{\lambda} =\chi \circ \lambda: \msf{V}_x \rightarrow \CC^{\times}$ be the composition of $\lambda$ with an injective morphism $k_F \rightarrow \CC^{\times}$. Lift $\chi_{\lambda}$ to a character of $G_{x, 1/2}$. $G_{x}$ acts on the set of characters of $G_{x, 1/2}$, and we denote by $H_{x, \lambda}$ the stabilizer of $\chi_{\lambda}$.

Let $C_{\lambda} = H_{x, \lambda} / G_{x, 1/2}$, which is a finite abelian $2$-group.
We denote by $\psi$ any character of $C_{\lambda}$. Each such character yields a lift $\chi_{\lambda, \psi}$ of $\chi_{\lambda}$ to $H_{x, \lambda}$.

We let $\rho_{\lambda, \psi} := \ind_{H_{x,\lambda}}^{G(F)} \chi_{\lambda, \psi}$ be the irreducible, epipelagic representation induced by $\chi_{\lambda, \psi}$.

Similarly, given another stable functional $\lambda'$ and another character $\psi'$ of $C_{\lambda'}$, we denote $\rho_{\lambda', \psi'} := \ind_{H_{x,\lambda'}}^{G(F)} \chi_{\lambda', \psi'}$.

\subsubsection{Parametrization of representations}
If $\lambda$ and $\lambda'$ are conjugate under $\wt{\msf{G}}_x$ (say, by an element $\widehat s \in \wt{\msf{G}}_x$), we obtain a canonical identification $C_{\lambda} = C_{\lambda'}$ (given by conjugation by any lift of $\widehat s$).

$\wt{\msf{G}}_x$ is isomorphic to the split orthogonal group $\mrm{O}_{n}$, and its connected component $\mrm{SO}_{n}$ has $k_F$-points corresponding to $G_{x,0} / G_{x,1/2}$.
As representations of $\mrm{SO}_{n}$, we have $\check{\msf{V}}_x \cong \msf{V}_x \cong \Sym^2(\mrm{Std})$, where the first isomorphism is given by the trace form.

This allows us to identify $\lambda$ with a $n \times n$ self-adjoint matrix with coefficients in $k_F$. The stability condition implies that this matrix is regular semisimple. Let $p_{\lambda}(x) \in k_F[x]$ be its characteristic polynomial, and $p_{\lambda}(x) =  \prod_{i=1}^m p_i(x)$ be its factorization into irreducibles, with $p_i(x) \in k_F[x]$ having degree $d_i$.
Let also $\delta_i \in k_{d_i}$ be any root of $p_i(x)$, where $k_{d_i}$ denotes the unique 	extension of $k_F$ of degree $d_i$.
Letting $D_{\lambda} := \Res_{(k_F[x]/p_{\lambda}(x))/k_F} \mu_2$, we notice that we have $C_{\lambda} \cong D_{\lambda}(k_F)$. Then $C_{\lambda} \cong \prod_i \mu_2(F_i)$, where $F_i$ is the unique unramified extension of $F$ with degree $d_i$, and we let $c_i \in C_{\lambda}$ to be the element corresponding to $(0, \ldots, 0, \underbrace{1}_i, 0, \ldots, 0)$ in this isomorphism.

Let $\wt{ \lambda}$ be a fixed lift of $\lambda$. Then the factorization of $p_{\lambda}(x)$ induces a splitting of $(V,\langle \cdot, \cdot \rangle)$ into eigenspaces for the unitary operator $\wt {\lambda}$: \[ (V , \langle \cdot, \cdot \rangle) \cong \bigoplus_{i=1}^m  (E_i, \langle \cdot, \cdot \rangle_i) .
\]
where $E_i$ is the unramified extension of $E$ of degree $d_i$. In fact, this splitting holds at the integral level (as $\Cal{O}_E$-lattices), and upon reducing modulo $\varpi_E$, we denote by $\overline V$ the corresponding $k_F$-vector space with the symmetric form obtained upon reduction.

The non-degeneracy of the trace pairing implies that there exists $\eta_i \in F_i^{\times}$ (where we recall that $F_i$ is the unramified extension of $F$ of degree $d_i$) such that $\langle x, y \rangle_i = \Tr_{E_i/E} (\eta_i x \ol{y})$ where $\ol{y}$ is the conjugate of $y$ under the non-trivial involution in $\Gal(E_i/F_i)$.
Each $\eta_i$ is defined up to $N_{E_i/F_i} (E_i^{\times} )$, and thus we can identify $\eta_i \in F_i^{\times} / N_{E_i/F_i} (E_i^{\times} ) \cong k_F^{\times} / (k_F^{\times})^2$.

Let $\varepsilon_{\lambda}: C_{\lambda} \rightarrow \CC^{\times}$ be the character defined by \[ \varepsilon_{\lambda}(c_i) =   \leg{-1}{q}^{d_i(n-d_i)}  \cdot  \left( \prod_j (-1)^{d_j-1}  \right)^{d_i} \cdot \leg{\eta_i}{q^{d_i}}^{n} \cdot \leg{D}{q}^{d_i}, \] where $D$ is the discriminant of $\overline V$.

Letting $\lambda'$ be another stable functional in $\msf{V}_x$, we obtain as before $d_i', \eta_i', \delta_i'$ and $\varepsilon_{\lambda'}$.

Define the character $r_{\lambda,\lambda'}: C_{\lambda} \rightarrow \CC^{\times}$ by \[ r_{\lambda, \lambda'} (c_i) = \prod_{i \neq j}   \prod_{t=1}^{\gcd(d_i,d_j)} \leg{\delta_i - \delta_j^{q^{t-1}}}{q^{[d_i,d_j]}} /  \leg{\delta_i' - (\delta_j')^{q^{t-1}}}
{q^{[d_i,d_j]}}. \]

Our most general result is the following.
\begin{lemma}[Lemma \ref{lemmamostgeneral}] The epipelagic representations $\rho_{\lambda, \psi}$ and $\rho_{\lambda', \psi'}$ have the same Langlands parameter if and only if the following equation in $C_{\lambda}^{\vee}$ holds: \[ \varepsilon_{\lambda} - \varepsilon_{\lambda'} + r_{\lambda, \lambda'} = \psi' - \psi. \]
\end{lemma}

To simplify this and bring $\lambda$ itself into the picture, we show that $C_{\lambda}^{\vee} \cong H^1(k_F, D_{\lambda})$. 
The $k_F$-rational orbits of $\lambda$ within its stable orbit (which is the intersection of $\msf{V}_{x}(k)$ with the orbit of $\lambda$ under $\SO_{n}(\ol{k})$ in $\msf{V}_{x} \otimes \ol{k}$) are a torsor for $H^1(k_F, D_{\lambda})$, and therefore choosing a basepoint allows us to interpret $\lambda$ as an  element of $C_{\lambda}^{\vee}$.

Our basepoint is the Kostant-Weierstress section $\lambda_{KW}$: this is a subspace of $\msf{V}_x$ which projects isomorphically down to $\msf{V}_x // \mrm{SO}_{n}$.
Explicitly, the image of $\lambda_{KW}$ consists of matrices of the form 
\[
 \left\{\begin{pmatrix} 
a_1  & \ldots & a_{n-2} &  a_{n-1} & a_{n} \\ 
1   & & & &  a_{n-1} \\
 &1   & & &  a_{n-2}  \\
& & \ddots  & & \vdots \\
&  & & 1 & a_1
\end{pmatrix}  \right\}
\]
where all the unlabeled entries are $0$.

We obtain then the following result:
\begin{prop}[Proposition \ref{kostant identity}]
The element $\left[ \lambda - \lambda_{KW} \right]$ of $C_{\lambda}^{\vee}$ is the character defined by \[ c_i \mapsto  (-1)^{d_i-1} \leg{(-1)^{\lfloor d_i/2\rfloor}}{q}  \leg{\eta_i}{q^{d_i}} \prod_{j \neq i}\prod_{t=1}^{(d_i,d_j)} \leg{\delta_i - \delta_j^{q^{t-1}}}{q^{[d_i,d_j]}}. \]
\end{prop}

In case of odd $n$, we obtain considerable cancellations among the terms involved, as our last two results show.
\begin{thm}[Theorem \ref{unitary epi param}]
Consider $U_n$ with $n$ odd. The epipelagic representation $\rho_{\lambda, \psi}$ and $\rho_{\lambda', \psi'}$ lie in the same $L$-packet if and only if $[\lambda-\lambda']  = \psi'- \psi \in C_{\lambda}^{\vee}$.
\end{thm}

Let $z \in C_{\lambda}^{\vee}$ be the character defined by $z(c_i) = -1$ for all $i$, which corresponds to the diagonal matrix $\mrm{diag}(-1, \ldots, -1)$. 

\begin{cor}[Corollary \ref{special epi param}]
Consider $\SU_n$ with $n$ odd. The epipelagic representation $\rho_{\lambda, \psi}|_{\SU_n}$ and $\rho_{\lambda, \psi}|_{\SU_n}$ lie in the same $L$-packet if and only if 
\[
[\lambda-\lambda'] =  \psi' - \psi   \in (C_{\lambda} \cap \SU_n)^{\vee} \quad \text{or} \quad 
[\lambda-\lambda'] =  z + \psi' - \psi  \in (C_{\lambda} \cap \SU_n)^{\vee} .
\] 
\end{cor}

\subsection{Acknowledgments}

We thank Tasho Kaletha for his encouragement to write up this paper. We are also grateful to Tasho, Mark Reeder, and Beth Romano for helpful discussions. This document benefited from comments and corrections by Tasho Kaletha, Paul Levy, Beth Romano, and the anonymous referee. During the preparation of this paper, TF was supported by an NSF Graduate Fellowship.. 

\section{Notation}\label{sec: notation}

We collect here some notation which will be used frequently throughout the paper. 

\begin{itemize}
\item We fix a local field $F$, of residue characteristic $p>2$. We write $W_F$ for the Weil group of $F$, $\Gamma_F = \Gal(\ol{F}/F)$, $I_F \triangleleft \Gamma_F$ for the inertia subgroup, and $P_F \triangleleft I_F$ for the wild inertia subgroup. 

\item Let $E/F$ be a ramified quadratic extension. We will sometimes denote the Galois conjugation of $E$ over $F$ by $e \mapsto \ol{e}$. Let $\varpi_F$ be a uniformizer for $F$. For later convenience we choose $\varpi_E$ so that its conjugate over $F$ is precisely $-\varpi_E$. 

\item For $K$ a local field or finite field, we will often denote by $K_d$ the unique unramified extension of $K$ of degree $d$. 

\item We denote by $k_F$ the residue field of $F$, and similarly for other local fields. Fix a faithful additive character $\chi \co k_F \rightarrow \CC^{\times}$. 

\item For $T$ a torus in a reductive group $G$ over $F$, we denote by $R(T,G)$ the set of roots for $S$ in $G$, and $\Omega(T,G)$ the Weyl group of $G$ relative to $T$. 
\end{itemize}

Throughout the paper we will conflate notation for an algebraic group over $p$-adic field with that of its points. 

\section{Epipelagic representations}\label{section epipelagic representations} 

We want to define certain epipelagic representations for $\SU_n$ or $\mrm{U}_n$ over $F$. These representations are built out of induction from certain maximal compact subgroups of characters which define ``stable functionals in the Moy-Prasad filtration''. We recall the essentials of the definition and notation from \cite{RY14} \S 2.

\subsection{The Moy-Prasad filtration}\label{moy prasad filtration}
Let $G$ be a reductive group over $F$. To a point $x$ in the Bruhat-Tits building $\Cal{B}_G(F)$, Moy and Prasad \cite{MP94} attached a filtration $\{G_{x, r}: r \in \R\}$ of $G(F)$. Similarly, at the level of Lie algebras we have a Moy-Prasad filtration $\{\mf{g}_{x,r} \co r \in \R\}$ such that $\mf{g}_{x,r+1} = \varpi_F \mf{g}_{x,r}$ for all $r$. We do not recall the definition here; see \cite[\S 4]{RY14} for a reference. However we note that for $x \in \cal{B}_G(F)$, $G_{x,0}$ is the parahoric group attached to $x$ by Bruhat-Tits theory; it is contained with finite index inside the stabilizer $G_x$ of $x$.

We will take $x$ to be a point of $\Cal{B}_G(F)$ which becomes a hyperspecial vertex in $\Cal{B}_G(E)$, but is \emph{not} already a hyperspecial vertex in $\Cal{B}_G(F)$. Then $G_{x,0}= \Cal{G}(\Cal{O}_F)$ for the smooth $\Cal{O}_F$-model $\Cal{G}$ for $G$ coming from Bruhat-Tits theory. We denote $G_{x,r+ } := \bigcup_{s >r} G_{x,s}$. For $r > 0$, there is a canonical isomorphism $G_{x,r}/G_{x,r+} \cong \mf{g}_{x,r} / \mf{g}_{x,r+}$.

We will now specialize this discussion to $G=\mrm{U}_n$, the special unitary group associated with the extension $E/F$ defined by the standard hermitian product on the vector space $E^{\oplus n}$: 
\begin{equation}\label{hermitian form}
\langle (x_1, \ldots, x_n), (y_1,\ldots, y_n) \rangle = x_1 \ol{y}_n + x_2 \ol{y}_{n-1} +  \ldots + x_n \ol{y}_1. 
\end{equation}
Then the jumps in the Moy-Prasad filtration at $x$ occur at $\frac{1}{2} \Z$. We will also be interested in the group $\SU_n$, for which the analogous statements follow easily from the $\mrm{U}_n$ case. 

\begin{remark}\label{rem: MP filtration} In terms of the integral model $\Cal{G}$, we can think of $G_{x,r/2}$ as the congruence subgroup given by the kernel of reduction mod $\varpi_E^r$: 
\[
G_{x,r/2} = \ker \left(\Cal{G}(\Cal{O}_F ) \rightarrow \Cal{G}(\Cal{O}_F )  \mod{\varpi_E^r} \right)
\]
where the notation ``$\Cal{G}(\Cal{O}_F ) \mod{  \varpi_E^r}$'' means ``consider $\Cal{G}(\Cal{O}_F)$ as a subgroup of  $\cal{G} (\Cal{O}_E)$, and reduce mod $\varpi_E^r$''. A similar remark applies to the filtration on the Lie algebra. 
\end{remark}

\begin{remark}
The group $ G_{x,0}/G_{x,0+}$ acts by conjugation on each $G_{x,r}/G_{x,r+}$ and $\mf{g}_{x,r}/\mf{g}_{x,r+}$. We denote by  $\wt{\msf{G}}_x$ the algebraic group over $k_F$ underlying $G_{x,0}/G_{x,0+}$.
\end{remark}

We use the trace form to identify $\mf{g}$ with its dual, which descends to an identification $\mf{g}_{x,r} / \mf{g}_{x,r+} \cong (\mf{g}_{x,-r}/\mf{g}_{x,-r+})^{\vee}$.
In particular, we obtain an identification $\mf{g}_{x,1/2} / \mf{g}_{x,1} \cong (\mf{g}_{x,-1/2}/\mf{g}_{x,0})^{\vee}$. On the other hand, multiplication by $\varpi_F$	 also defines an isomorphism $\mf{g}_{x,-1/2}/\mf{g}_{x,0} \xrightarrow{\sim} 
\mf{g}_{x,1/2} / \mf{g}_{x,1}$. Thus, the trace form induces an isomorphism 
\begin{equation}\label{killing self dual}
\mf{g}_{x,1/2} / \mf{g}_{x,1}\cong (\mf{g}_{x,1/2} / \mf{g}_{x,1})^{\vee}
\end{equation}
given by the bilinear form 
\begin{equation}\label{functional duality convention}
X,Y \mapsto \Tr( \varpi_F^{-1} XY).
\end{equation}
This is evidently equivariant for the conjugation action of $G_{x,0}/G_{x,1/2} \cong \SO_n(k_F)$, where the isomorphism is to be proven in Lemma \ref{quotient orthogonal group}.

\subsection{Construction of epipelagic representations}\label{section construction epip reps}
 
 Assume now that $G$ is a tamely ramified quasi-split reductive group over $F$. For a point $x \in \Cal{B}_G(F)$, we denote by $r(x)$ the largest positive number for which $G_{x, r(x)} = G_{x,0+}$. In our case of interest, $r(x)=1/2$. 

\begin{defn}[{\cite[\S 2.5]{RY14}}]
An irreducible representation $\pi$ of $G(F)$ is \emph{epipelagic} if $\pi$ has depth $r(x)$ and a non-zero vector invariant under $G_{x, r(x)+}$. 
\end{defn}

We review the construction of epipelagic representations from \cite[\S 2]{RY14}. At this point we must note that \cite[\S 2]{RY14} is formulated under the hypothesis that $G$ is also \emph{semisimple}, which is the case for $\SU_n$ but not $\mrm{U}_n$. However, the proofs of the statements below are valid for reductive groups without any change to the proofs. (\cite{Kal15} is formulated in the generality that we work with here.) It would be possible to carry our the entire computation of this paper just for $\SU_n$, which is what we originally tried to do, but it is actually very useful at several points in the calculation to use that the theory extends to $\mrm{U}_n$.

 Let $\wt{\msf{G}}_x$ be the algebraic group over $k_F$ underlying $G_x/G_{x,r(x)}$ and $\msf{G}_x$ be the algebraic group over $k_F$ underlying $G_{x,0}/G_{x,r(x)}$. The paper \cite{RY14} is phrased using $\msf{G}_x$, but for our purposes it is more convenient to work with $\wt{\msf{G}}_x$. Let $\msf{V}_x$ the algebraic representation of $\wt{\msf{G}}_x$ over $k_F$ underlying $G_{x,1/2}/G_{x,1}$. 

To build an epipelagic representation, we need to start with a functional $\lambda$ on $\msf{V}_x$ which is \emph{stable} for the $\wt{\msf{G}}_x$ in the sense of geometric invariant theory (GIT), meaning that it has finite stabilizer and closed orbit. Then we inflate the composition $\chi \circ \lambda \co \msf{V}_x(\mf{f}) \rightarrow \CC^{\times}$ to a character $\chi_{\lambda}$ of $G_{x, r(x)}$, and consider the compact induction
\[
\pi_x(\lambda) := \ind_{G_{x,r(x)}}^{G(F)} \chi_{\lambda}.
\]

\begin{prop}[{\cite[Proposition 2.4]{RY14}}]\label{epipelagic representation}
The representation $\pi_x(\lambda)$ is a finite direct sum of irreducible epipelagic representations.
\end{prop}

We can be a little more precise about the summands appearing in $\pi_x(\lambda)$. Let $H_{x,\lambda} \subset G_x$ be the stabilizer of the character $\chi_{\lambda}$ on $G_{x,r(x)}$. 
As explained in \cite[\S 2.1]{RY14}, Mackey Theory implies that $\ind_{G_{x,1/2}}^{H_{x,\lambda}} \chi_{\lambda}$ is a direct sum of representations $\chi_{\lambda, \psi}$, where $\psi$ indexes the representations of $H_{x,\lambda}/G_{x,1/2}$. We let 
\[
\rho_{\lambda, \psi} := \ind_{H_{x,\lambda}}^{G(F)} \chi_{\lambda, \psi}.
\]

\begin{remark}
Note that the notion of a stable vector for $\wt{\msf{G}}_x$ coincides with that for $\msf{G}_x$, since the latter is a finite-index subgroup of the former. (The condition of finite stabilizer is obviously unchanged by passing to finite index subgroups, and orbits break up into finite disjoint unions, so the condition of an orbit being closed is also unchanged.) 
\end{remark}

\begin{prop}[\cite{RY14}, Proposition 2.4]\label{irreducible epipelagic}
The representation $\rho_{\lambda, \psi} $ is epipelagic.
\end{prop}

We note that the epipelagic representations obtained from $\lambda$ depend only on the conjugacy class of $\lambda$. The following lemma is undoubtedly well-known, but we  did not find its statement in the literature.

\begin{lemma}\label{lem: isom iff rat conj}
If $\lambda$ and $\lambda' $ are stable functionals (in the GIT sense) on $\msf{V}_x$ which are conjugate under $\wt{\msf{G}}_x$, then $\pi_x(\lambda) \cong \pi_x(\lambda')$. Conversely, if there exist $\psi, \psi'$ such that $\rho_{\lambda,\psi} \cong \rho_{\lambda', \psi'}$, then $\lambda$ is conjugate under $\wt{\msf{G}}_x$ to $\lambda'$. 
\end{lemma}

\begin{proof}
Let $H_{\lambda}$ (resp. $H_{\lambda'}$) be the stabilizer of $\lambda$ (resp. $\lambda'$) in $G_x$, and $\chi_{\lambda}$ (resp. $\chi_{\lambda'}$) the character of $H_{\lambda}$ (resp. $H_{\lambda'}$) used to induce $\rho_{\lambda, \psi}$ (resp. $\rho_{\lambda', \psi'}$). By Mackey Theory, we have $\rho_{\lambda,\psi} \cong \rho_{\lambda', \psi'}$ if and only if 
\[
\bigoplus_{s \in H_{\lambda'} \backslash G(F) / H_{\lambda}} \Hom_{H_{\lambda}} (\chi_{\lambda}, \Ind_{s^{-1} H_{\lambda'} s \cap  H_{\lambda}}^{H_{\lambda}} \chi_{\lambda'}^s ) \neq 0.
\]
Hence if the stable functionals are conjugate then there exists $s \in \wt{\msf{G}}_x$ such that $s^{-1} H_{\lambda'} s = H_{\lambda}$ and $\chi_{\lambda'}^s = \chi_{\lambda}$. 

The converse follows from \cite[Fact 3.8]{Kal15}. 
\end{proof}

\begin{remark}
We see from the proof of Lemma \ref{lem: isom iff rat conj} that it makes sense to identify $C_{\lambda}$ and $C_{\lambda'} $ for rationally conjugate $\lambda$ and $\lambda'$, and thus to compare $\psi \in C_{\lambda}$ and $\psi' \in C_{\lambda'} := H_{\lambda}/G_{x,1/2}$.
Notice, indeed, that if $s, \widehat s \in \wt{\msf{G}}_x$ are such that $\chi_{\lambda'}^s = \chi_{\lambda} = \chi_{\lambda'}^{\widehat s}$, then $\widehat s^{-1} s$ fixes $\chi_{\lambda}$ and thus belongs to $H_{\lambda}$.
By the discussion following Lemma \ref{symmetric square rep} below, $C_{\lambda} = H_{\lambda} / G_{x, 1/2}$ is abelian, and thus computing in $H_{\lambda} / G_{x, 1/2} \subset G_x / G_{x, 1/2}$ gives that \[ \forall h_{\lambda} \in H_{\lambda} / G_{x, 1/2}, \quad \widehat s^{-1} s h_{\lambda} \left( \widehat s^{-1} s \right)^{-1} = h_{\lambda} \Rightarrow s h_{\lambda} s^{-1} = \widehat s h_{\lambda} \widehat s^{-1} \in H_{\lambda'} / G_{x, 1/2} = C_{\lambda'}. \]
\end{remark}

We now specialize this discussion to our case of interest to construct certain epipelagic representations. The first task is to find some stable functionals, so we need to identify the representation in question.

\begin{lemma}\label{quotient orthogonal group}
For $G = \mrm{U}_n$ and the choice of $x$ as above, we have $\wt{\msf{G}}_x \cong \mrm{O}_{n}$, the split orthogonal group, and $\msf{G}_x \cong \mrm{SO}_{n}$.
\end{lemma}

\begin{proof}
Using the description in Remark \ref{rem: MP filtration}, we may identify $G_{x}/G_{x,1/2}$ with the group of automorphisms of $k_E = k_F$ preserving the split quadratic form 
\begin{equation}\label{eqn: orthogonal form}
\langle (\ol{x}_1, \ldots, \ol{x}_n), (\ol{y}_1,\ldots, \ol{y}_n) \rangle = \ol{x}_1 \ol{y}_n + \ol{x}_2 \ol{y}_{n-1} +  \ldots +\ol{x}_n \ol{y}_1. 
\end{equation}
It only remains to note that $G_{x,0}/G_{x,1/2} $ is the connected component.
\end{proof}

\begin{lemma} \label{symmetric square rep}
For the choice of $x$  as above, we have $\msf{V}_x \cong  \Sym^2(\mrm{Std})$ as representations of $\mrm{O}_n$. 
\end{lemma}

\begin{proof}
As a representation of $\mrm{O}_n$, we view $\Sym^2(\mrm{Std})$ as the space of $n \times n$ self-adjoint matrices (with respect to the form \eqref{eqn: orthogonal form}) over $k_E = k_F$, with the conjugation action of $\mrm{O}_n$. Using Remark \ref{rem: MP filtration}, we can view $\msf{V}_x $ as the space of $n \times n$ matrices over $k_E$ preserving the form \eqref{hermitian form}. Since the conjugation of $E$ over $F$ induces multiplication by $-1$ on 
\[
\varpi_E \Cal{O}_E / \varpi_E^2 \Cal{O}_E \cong \Cal{O}_E/\varpi_E = k_E 
\]
(cf. \S \ref{sec: notation}) the condition of preserving the form \eqref{hermitian form} translates to the self-adjointess condition $-A^{\dagger} + A = 0$.
\end{proof}

Since \eqref{killing self dual} furnishes an identification $\check{\msf{V}}_x \cong \msf{V}_x$ a stable functional on $\msf{V}_x$ is the same as a stable vector in $\msf{V}_x$. Then $\lambda \in \msf{V}_x$ is stable if and only if it is regular semisimple when viewed as a self-adjoint matrix (cf. \cite[\S 6]{BG14} for this statement, although this was undoubtedly known long ago). Abusing notation, we also let $\lambda$ denote the functional on $V_{x,1/2}$ corresponding to $\lambda$ under the trace form, and $\chi_\lambda$ the corresponding character of $G_{x,1/2}$.

The epipelagic representations associated to $\chi_{\lambda}$ are summands of $\ind_{G_{x,1/2}}^{G(F)} \chi_{\lambda}$. To parametrize them explicitly, we identify the stabilizer of $\chi_{\lambda}$ in $G_x$. Viewing $\lambda \in \Sym^2 \msf{V}_x$ as a self-adjoint matrix, the regularity of $\lambda$ implies that any $n \times n$ matrix commuting with $\lambda$ is a polynomial in $\lambda$. If $p_{\lambda}(x)$ denotes the characteristic polynomial of $\lambda$, the space of matrices which can be expressed as a polynomial in $\lambda$ is $k_F[x]/p_{\lambda}(x)$. Since any such matrix is self-adjoint, it is furthermore orthogonal if and only if it squares to $1$. Hence the stabilizer of $\lambda$ within $\mrm{O}_n$ is the group scheme
\[
D_{\lambda} := \Res_{(k_F[x]/p_{\lambda}(x))/k_F} \mu_2.
\]
We let $C_{\lambda} = D_{\lambda}(k_F)$: this is the stabilizer of $\lambda$ in $\wt{\msf{G}}_x(k_F)$. This discussion shows that $H_{x, \lambda}$ fits into an extension
\[
0 \rightarrow G_{x,1/2} \rightarrow H_{x, \lambda} \rightarrow  C_{\lambda} \rightarrow 0.
\]
This sequence admits a splitting since $G_{x,1/2}$ is pro-$p$ while $|C_{\lambda}|$ is a power of $2$ (Schur-Zassenhaus Theorem), hence we can write 
\[
H_{x, \lambda} = G_{x,1/2} \rtimes  C_{\lambda}.
\]
Since $C_{\lambda}$ acts trivially on $\lambda$, we may extend $\chi_{\lambda} $ to a character of the semidirect product $H_{x, \lambda}$. The possible extensions are parametrized precisely by the characters of $C_{\lambda}$: for any character $\psi$ of $C_{\lambda}$, we have a character $\chi_{\lambda} \cdot \psi$ of $H_{x, \lambda}$ defined by 
\begin{equation} \label{eqn: extension of chi}
(\chi_{\lambda} \cdot \psi)(g,c) = \chi_{\lambda}(g) \psi(c).
\end{equation}

Now we can describe the components of $\ind_{G_{x,1/2}}^{G(F)} \chi_{\lambda}$, according to Proposition \ref{irreducible epipelagic}. 

\begin{lemma}
We have 
\[
\ind_{G_{x,1/2}}^{ H_{x,\lambda}} \chi_{\lambda} \cong \bigoplus_{\psi  \in\wh{C}_{\lambda}} \chi_{\lambda} \cdot \psi.
\] 
\end{lemma}

\begin{proof}
By Frobenius reciprocity, for each character $\psi$ of $C_{\lambda}$ we have 
\[
\Hom_{H_{x,\lambda}}(\chi_{\lambda} \cdot \psi, \ind_{G_{x,1/2}}^{ H_{x,\lambda}} \chi_{\lambda}) = \Hom_{G_{x,1/2}}(\chi_{\lambda}, \chi_{\lambda}) 
\]
so every character of the form $\chi_{\lambda} \cdot \psi$ can be embedded into $\ind_{G_{x,1/2}}^{H_{x,\lambda_T}} \chi_{\lambda}$. Since these characters are all distinct, dimension-counting shows that they must fill up the entire induced representation, so we deduce the result. 
\end{proof}

\begin{cor}
The (irreducible) epipelagic representations corresponding to $\lambda$ are the representations 
 $\ind_{H_{x,\lambda}}^{G(F)} \chi_{\lambda} \cdot \psi$ as $\psi$ ranges over characters of $C_{\lambda}$.
 \end{cor}

Let $\rho_{\lambda, \psi} = \ind_{H_{x,\lambda}}^{G(F)} \chi_{\lambda} \cdot \psi$. Now we note that the same discussion applies for $\SU_n$, using the same $\lambda$ to produce an epipelagic representation. The following Lemma says that the restriction of $\rho_{\lambda, \psi}$ to $\SU_n(F)$ is already irreducible, so that it necessarily coincides with an epipelagic representation of $\SU_n(F)$ built from $\lambda$. 

\begin{lemma}\label{SU_n restriction}
Each representation $\rho_{x, \lambda}|_{\SU_n(F)}$ is already irreducible, and we have 
\[
\rho_{x, \lambda}|_{\SU_n(F)} \approx \rho_{x, \lambda'}|_{\SU_n(F)}
\]
if and only if $\lambda$ and $\lambda'$ agree on $\SU_n(F)\cap C_{\lambda}$. 
\end{lemma}

\begin{proof}
The second claim follows immediately from the first and \cite[Proposition 2.4(2)]{RY14}. To prove the first claim, by Mackey's formula and Proposition \ref{irreducible epipelagic} it suffices to see that 
\[
\# \SU_n(F) \backslash \mrm{U}_n(F) / H_{x,\lambda}  = 1. 
\]
or in other words that $H_{\lambda}(F)$ surjects onto $\SU_n(F) \backslash \mrm{U}_n(F)  = \mrm{U}_1(E/F)$. Since $\lambda$ was chosen to be regular semisimple, any lift  of it to $G_{x, 1/2}$ is a polynomial with distinct roots, so its centralizer (which is contained in $H_{x,\lambda}$) is a torus of the form $\prod \mrm{U}_1(E_i/F_i)$ with $E_i,F_i$ unramified extensions of $E,F$ respectively. The proof is then concluded by recalling that the norm map for an unramified extension $E_i/E$ is surjective onto $\mrm{U}_1(E/F)$ when restricted to $\mrm{U}_1(E_i/F_i)$.  
\end{proof}

\section{Langlands correspondence for epipelagic representations}\label{section Langlands parameters} 

We review the construction of the local Langlands correspondence for epipelagic representations in \cite{Kal15}. Let $G$ a tamely ramified quasi-split reductive group over $F$. 

\begin{defn}[{\cite[p. 37]{Kal15}}] \label{epipelagic parameter}
An \emph{epipelagic parameter} for $G$ over a local field $F$ is a Langlands parameter 
\[
\varphi \co W_F \rightarrow {}^L G
\]
satisfying the following conditions:
\begin{enumerate}
\item $\wh{T} = C(\varphi(P_F),\wh{G})$ is a maximal torus of $\wh{G}$.
\item The image of $\varphi(I_F)$ in $\Omega(\wh{T}, \wh{G}) \rtimes  I_F$ is generated by a regular elliptic element. 
\item If  $w \in I_F^{1/m+}$, where $m$ is the order of the regular elliptic element, then $\varphi(w)=(1,w)$.
\end{enumerate}

\end{defn}

The Langlands correspondence predicts that to a Langlands parameter $\varphi \co W_F \rightarrow {}^L G $ there should correspond an L-packet $\Pi_{\varphi}$ of representations of $G(F)$. Kaletha constructed this correspondence for epipelagic parameters in \cite{Kal15}, and we summarize the description of the $L$-packets, following \cite[\S 5]{Kal15}. We should clarify here that when we say ``$L$-packet'' we mean only the constituents of the L-packets considered in \cite{Kal15} which are representations of our chosen $\mrm{U}_n$. In other words we are ignoring representations of inner forms, and considering only those representations of $\mrm{U}_n$ which lie in same $L$-packet in the sense  of \cite{Kal15}.

\subsection{Step one: factorization through an admissible embedding} 

Let $\wh{S}$ be the Galois-module whose underlying abelian group is the complex torus $\wh{T}$ and whose Galois action is furnished by the composite
\[
\varphi \co W_F \rightarrow N(\wh{T}, \wh{G}) \rtimes W_F \rightarrow \Omega(\wh{T}, \wh{G}) \rtimes W_F \rightarrow \Aut(\wh{T}).
\]
We will construct a particular $\wh{G}$-conjugacy class of embeddings ${}^L j \co {}^L S \rightarrow {}^L  G$ which are tamely ramified in the sense that ${}^L j (1,w) = (1,w)$ for all $w \in P_F$. In this conjugacy class there is an embedding such that ${}^L j(\wh{S}) = \wh{T}$, and such that the following two composite homomorphisms are equal:
\[
\begin{tikzcd}
W_F \ar[rr, "\varphi"]  & &  N(\wh{T}, \wh{G}) \rtimes W_F \ar[r] & \Omega(\wh{T}, \wh{G}) \rtimes W_F \\
W_F \ar[r, "\iota_2"]  & \wh{S} \rtimes W_F \ar[r, "{}^L j"] & N(\wh{T}, \wh{G}) \rtimes W_F \ar[r] & \Omega(\wh{T}, \wh{G}) \rtimes W_F
\end{tikzcd} 
\]
Since ${}^L j$ contains the image of $\varphi$, choosing such an embedding gives a factorization of $\varphi$ through a homomorphism $\varphi_{S, {}^L j} \co W_F \rightarrow {}^L S$ such that ${}^L j \co {}^L S \hookrightarrow {}^LG $ is an admissible homomorphism: 
\[
\begin{tikzcd}
W_F \ar[rr, "\varphi"] \ar[dr, "\varphi_{S, {}^L j}"'] &  & {}^L G \\
& {}^L S \ar[ur, hook, "{}^L j"'] 
\end{tikzcd}
\]

It is worth noting that there are many possibilities for the conjugacy class of admissible embedding ${}^L j \co {}^L S \hookrightarrow  {}^L  G$. Moreover, for epipelagic representations the choice of ${}^L j$ is really significant, in contrast to previous incarnations of this contruction (see discussion on \cite[p.3]{Kal15}). The correct conjugacy class is specified by a subtle construction of $\chi$-data from the particular parameter $\varphi$, as described in \cite[\S 5.2]{Kal15}. We will not go into the details here, leaving them for when we actually need to compute, in \S \ref{section admissible embedding}.

\subsection{Step two: the toral invariant}
By the local Langlands correspondence for tori, from the Langlands parameter $\varphi_{S, {}^L j} \co W_F \rightarrow {}^L S$ constructed in Step one, we obtain a character 
\[
\chi_{S,{}^L j} \co S(F) \rightarrow \CC^{\times}
\]
attached to $\varphi_{S, {}^L j}$. Then \cite[\S  3.3]{Kal15} describes a construction starting from a pair $(S(F), \chi)$ of a tamely ramified maximal torus and a character of $S(F)$, and producing an epipelagic representation of $G(F)$. We will elaborate on this in the next step. However, this construction is \emph{not} applied to $( S(F),\chi_{S,{}^L j})$: we first need to modify the character $\chi_{S,{}^L j}$ by a character 
\[
\epsilon_\varphi \co S(F) \rightarrow \CC^{\times}
\]
 obtained from the \emph{toral invariant} of \cite[\S 4]{Kal15}. Thus the second step is the computation of the toral invariant and the character $\epsilon_\varphi$. The toral invariant is a collection of characters of unit groups of local fields labelled by the roots of $G$. Again, we postpone the details until we actually need to compute it, in \S \ref{section toral invariant}.

\subsection{Step three: Local Langlands for tori}\label{step3}

For each admissible (cf. \cite[p.222]{LS87} for the definition) embedding $j \co S \hookrightarrow  G$, $S$ is an elliptic tamely-ramified  maximal torus, hence determines a point $x$ in the Bruhat-Tits building for $G$ \cite[Remark 3]{Pra01}. The pair $(j \co S(F) \hookrightarrow G(F),\chi_{S,{}^L j} \cdot \epsilon_j)$ satisfies certain conditions \cite[Conditions 3.3]{Kal15} allowing one to perform the construction of \cite[\S 3.3]{Kal15}. The construction goes as follows. Suppose we have a pair $j \co S \hookrightarrow G$ and $\chi \co S(F) \rightarrow \CC^{\times}$ satisfying \cite[Conditions 3.3]{Kal15}. The inclusion $j \co S \hookrightarrow G$ induces a decomposition 
\[
\mf{g} = \mf{s} \oplus \mf{n}
\]
where $\mf{n}$ is the sum of all isotypic subspaces on which $S$ acts nontrivially.
Hence we obtain 
\begin{equation}\label{3 torus decomposition}
\mf{g}(F)_{x,r} = \mf{s}(F)_{r} \oplus  \mf{n}(F)_{x,r}
\end{equation}
for all $x$ and $r$. The character $\chi$ factors through $S(F)_{2/e}$ where $e$ is the ramification degree of the field extension splitting $S$ (which is $2$ in our case of interest), and hence descends to a character of 
\[
S(F)_{1/e}/S(F)_{2/e} \cong \mf{s}(F)_{1/e}/\mf{s}(F)_{2/e}.
\]
By the decomposition \eqref{3 torus decomposition} we extend it to a character on 
\[
G(F)_{x, 1/e}/G(F)_{x,2/e} \cong \mf{g}(F)_{x,1/e}/\mf{g}(F)_{x,2/e}.
\]
Since the resulting character is stabilized by $S(F)$, it extends to
\[
\wh{\chi} \co S(F) G(F)_{x,1/e} \rightarrow \CC^{\times}.
\] 
Then $\ind_{G(F)_{x, 1/e} S(F)}^{G(F)} \wh{\chi}$ is an epipelagic representation for $G(F)$ (i.e. $G(F)_{x, 1/e} S(F)  = H_{x,\lambda}$ for an appropriate stable functional $\lambda$). The L-packet of $\varphi$ consists of the (epipelagic) representations
\[
\ind_{G(F)_{x, r(x)} S(F)}^{G(F)} \wh{\chi_{S, {}^L j}}  \epsilon_{\varphi}
\]
where $\epsilon_{\varphi}$ is the character obtained form the toral invariant, and $j$ ranges over admissible embeddings $j \co S \hookrightarrow G$.

\section{Calculation of Langlands parameters and $L$-packets}\label{section calculation}

\subsection{Overview}

By reversing Kaletha's construction explained in \S \ref{section Langlands parameters}, we will calculate the Langlands parameters of the epipelagic representations of $\mrm{U}_n$ and $\SU_n$ constructed in \S \ref{epipelagic representation}, and identify the L-packets. By \S \ref{section construction epip reps}, particularly Lemma \ref{SU_n restriction}, all the Langlands parameters for the relevant epipelagic representations of $\SU_n$ are obtained from those for $\mrm{U}_n$, so we are reduced to computing the latter. 

We will begin by identifying the (tamely ramified) anisotropic torus $S$ corresponding to the given point $x \in \Cal{B}_G(F)$. Then we will extract from the stable functional a character on $S(F)$, as discussed in \S \ref{step3}. At this point we must calculate the toral invariant and the corresponding character $\epsilon_f$ of $S$ (here the subscript $f$ depends on $S$, and stands for something that has not yet been explained), and modify the character by $\epsilon_f$. Then we will apply the local Langlands correspondence to obtain a Langlands parameter 
\[
W_F\rightarrow {}^L S.
\]
Finally, we will calculate the admissible embedding ${}^L j \co {}^L S \hookrightarrow {}^L G$ and compose the preceding Langlands parameter with it; the resulting composition 
\[
W_F\rightarrow {}^L S \xrightarrow{{}^L j} {}^L G
\]
is then the Langlands parameter we seek. 

In what follows, we retain the notation of \S \ref{epipelagic representation}. In particular, $G = \mrm{U}_n$ and $x$ is a point of $\Cal{B}_G(F)$ which becomes hyperspecial in $\Cal{B}_G(E)$. We have a regular semi-simple element $\lambda \in \msf{V}_x \cong \Sym^2(\mrm{Std})$, meaning that its characteristic polynomial viewed as a self-adjoint $n \times n$ matrix over $k_E = k_F$ has distinct roots in $\ol{k_E}$. The centralizer of $\lambda$ in $\wt{\msf{G}}_x$ is the group scheme 
\[
D_{\lambda} := \Res_{(k_F[x]/p_{\lambda}(x))/k_F} \mu_2,
\]
while the centralizer of $\lambda$ in $\wt{\msf{G}}_x(k_F)$ is denoted $C_{\lambda}$. The epipelagic representations $\rho_{\lambda, \psi}$ are parametrized by the choice of $T$ and a character $\psi$ of $C_{\lambda}$. 

\subsection{The anisotropic torus}


We now identify the (tamely ramified) maximal torus corresponding to the point $x$, in the sense of \cite[Remark 3]{Pra01}. Choose a lift $\wt{\lambda} \in \mf{g}_{x,1/2}$ of $\lambda$. Any such choice has characteristic polynomial with distinct roots, since the roots are even distinct after reduction, and hence is automatically regular semisimple. We may thus define a maximal torus $S \subset G$ such that $S(F) = Z_{G(F)}(\wt{\lambda})$.

\begin{lemma}\label{lemma torus}
The torus $S$ is tamely ramified and anisotropic, and corresponds to the point $x$.
\end{lemma}

\begin{proof}

Let $\wt{p}_{\lambda}(x)$ be the characteristic polynomial of $\wt{\lambda}$, viewed as an $n \times n$ matrix over $E$. Set $E_{\wt{\lambda}} =  E[x]/\wt{p}_{\lambda}(x)$. We have an embedding $E_{\wt{\lambda}}^{\times} \hookrightarrow \GL_n(E)$, and we may identify $\mrm{U}_n \cap E_{\wt{\lambda}}^{\times} = S(F)$. It is clear that $S$ splits over $E_{\wt{\lambda}}$, which is an unramified extension of $E$, hence tamely ramified over $F$. 

The conjugate transpose defines an involution on $E_{\wt{\lambda}}$, whose fixed field is a quadratic subfield $F_{\wt{\lambda}} \subset E_{\wt{\lambda}}$. The condition of being unitary corresponds to having norm $1$ under $\Nm_{E_{\wt{\lambda}}/ F_{\wt{\lambda}}}$. Therefore we see that
\[
S = \ker \left( \Nm_{E_{\wt{\lambda}}/ F_{\wt{\lambda}}} \co \Res_{E_{\wt{\lambda}}/F}  E_{\wt{\lambda}}^{\times}   \rightarrow \Res_{F_{\wt{\lambda}}/F} F_{\wt{\lambda	}}^{\times} \right)
\]
is anisotropic. 

Finally, we show that $S$ corresponds to $x$ in the building of $G$ over $F$. Let $E'_{\wt{\lambda}}$ be the splitting field of $\wt{p}_{\lambda}$, so $E'_{\wt{\lambda}} \supset E_{\wt{\lambda}}$. Viewing $\wt{\lambda}$ as a matrix, we can diagonalize it over $E'_{\wt{\lambda}}$, since it is regular semisimple. Moreover, we can pick the conjugating element to be in $G(\Cal{O}_{E'_{\wt{\lambda}}})$ by Hensel's Lemma (because $\wt{\lambda}$ is a lift of a regular semisimple matrix over the residue field). So $S$ is conjugate to the standard diagonal subgroup of $G(E'_{\wt{\lambda}}) \cong \GL_n(E'_{\wt{\lambda}})$ under $\GL_n(\Cal{O}_{E'_{\wt{\lambda}}})$. Therefore in the building of $G$  over $E'_{\wt{\lambda}}$, $S$ corresponds to the point represented by the maximal compact subgroup $G(\Cal{O}_{E'_{\wt{\lambda}}})$. Since $G(\Cal{O}_{E'_{\wt{\lambda}}})$ is stable under $\Gal(E'_{\wt{\lambda}}/F)$, the corresponding point in the building of $G$ over $F$ is represented by the maximal compact subgroup $G(F) \cap G(\Cal{O}_{E'_{\wt{\lambda}}})$ of $G(F)$, which by definition is the point $x$. 
\end{proof}

\begin{remark}
The torus $S$ depends on the choice of lift $\wt{\lambda}$ of $\lambda$. However, the eventual structural results about $L$-packets will be independent of this choice. 
\end{remark}

For later use, we record some more precise information about the torus $S$ that comes out of the proof of Lemma \ref{lemma torus}. We factor the characteristic polynomial as 
\[
p_{\lambda}(x) =  \prod_{i=1}^m p_i(x),
\]
where $p_i(x)$ is irreducible over $k_F$, of degree $d_i$. Let $E_i$ (resp. $F_i$) be the unramified extension of $E$ (resp $F$) of degree $d_i$. Then 
\begin{equation}\label{torus S}
S =\prod_{i=1}^m \ker \left( \Nm \co E_i^{\times} \rightarrow F_i^{\times} \right)
\end{equation}

\begin{remark}
The right hand side of \eqref{torus S} depends only on the $d_i$, which depend only on $\lambda$. This shows that $S$ is independent of the choice of lift $\wt{\lambda}$. 
\end{remark}

We fix notation for the character group. Let $S_i = \ker \left( \Nm \co E_i^{\times} \rightarrow F_i^{\times} \right)$. Then we may write 
\[
X^*(S_i) = \coker \left( \Z[F_i/F] \rightarrow \Z[E_i/F]  \right) 
\]
where the maps are  the ``diagonal'' embeddings, being dual to the norm. We may pick a basis for the cocharacter group such that 
\[
X^*(S_i)  = \frac{\Z [ \lambda^{(i)}_1, \ldots, \lambda^{(i)}_{d_i}, \ol{\lambda}^{(i)}_1, \ldots, \ol{\lambda}^{(i)}_{d_i} ]}{\langle \lambda^{(i)}_r +  \ol{\lambda}^{(i)}_r \mid r = 1,\ldots, d_i \rangle}.
\]
If $\sigma_i \in \Gal(E_i/F)$ denotes a lift of Frobenius and $\tau_i \in \Gal(E_i/F)$ denotes the involution with fixed field $F_i$, we can choose the basis such that the Galois action given by $\sigma_i \lambda^{(i)}_r = \lambda^{(i)}_{r+1}$ and $\tau_i \lambda^{(i)}_r = \ol{\lambda}_r^{(i)}$, and the roots of $S$ are
\[
\alpha^{(i,j)}_{r,s} := \lambda^{(i)}_r  - \lambda^{(j)}_s 
\]
where if $i=j$ then $r \neq s$. Therefore, the character group $X^*(S)$ can be described as 
\begin{equation}\label{torus character group}
X^*(S) = \bigoplus_{i=1}^m   \coker \left( \Z[F_i/F] \rightarrow \Z[E_i/F]  \right)
\end{equation}

\subsection{The toral invariant} \label{section toral invariant}
Let $G$ be a reductive group defined over a local field $F$, and $S \subset G$ a maximal torus defined over $F$. Let $R(S,G)$ be the set of roots of $G$ with respect to $S$. We attach a \emph{toral invariant} to the pair $(S,G)$ following \cite[\S 4]{Kal15}. The toral invariant is a function $f \co R(S,G) \rightarrow \{\pm 1\}$, and enters into the local Langlands correspondence via an attached character $\epsilon_f \co S(F) \rightarrow \CC^{\times}$ that we will define.

\subsubsection{Definition of the toral invariant} 
We first recall the definition of the toral invariant from \cite[\S 4]{Kal15}. The set of roots $R(S,G)$ carries an action of $\Gamma := \Gal(\ol{F}/F)$. 

\begin{defn}
An orbit of the $\Gamma$-action on $R(S,G)$ is \emph{symmetric} if it is preserved by multiplication by $-1$. Otherwise the orbit is called \emph{asymmetric}. 

If $I \subset \Gamma$ denotes the inertia group, then every $\Gamma$-orbit decomposes into a disjoint union of $I$-orbits, which have the property that they are either all preserved by multiplication by $-1$ or none are, in which case we call them \emph{inertially symmetric} or \emph{inertially asymmetric} respectively. 

A root $\alpha \in R(S,G)$ is called \emph{(inertially) symmetric or asymmetric} if its orbit is. We define
\[
\Gamma_{\alpha} := \Stab_{\Gamma}(\alpha), \quad \text{and} \quad 
\Gamma_{\pm \alpha} = \Stab_{\Gamma}(\{\alpha, -\alpha\}).
\]
Obviously $[\Gamma_{\pm \alpha}  \co \Gamma_{\alpha} ] = 1$ if $\alpha$ is asymmetric, and $[\Gamma_{\pm \alpha}  \co \Gamma_{\alpha} ]  = 2$ if $\alpha$ is symmetric. Let $F_{\alpha} \supseteq F_{\pm \alpha}$ be the corresponding fixed fields.
\end{defn}

We may now define the toral invariant, following \cite[\S 4.1]{Kal15}. Let $S \subset G$ be a maximal torus and $R(S,G)$ the roots of $G$ with respect to $S$. The \emph{toral invariant} is a function
\[
	f : = f_{(G,S)} \co R(S,G) \rightarrow \{ \pm  1\}
\]
defined as follows. If $\alpha \in R(S,G)$ is asymmetric, then $f(\alpha)=1$. Suppose $\alpha \in R(S,G)$ is a symmetric root. We have a corresponding one-dimensional root subspace $\mf{g}_{\alpha} \subset \mf{g}$ defined over $F_{\alpha}$. Let $H_{\alpha} = \operatorname d \alpha^{\vee} (1) \in \mathfrak s (F_{\alpha})$ be the coroot corresponding to $\alpha$, and choose $X_{\alpha} \in \mf{g}_{\alpha}(F_{\alpha})$. Let $\tau_{\alpha} \in \Gamma_{\pm \alpha} \setminus \Gamma_{\alpha}$. Then $\tau_{\alpha} X_{\alpha}$ is a non-zero element of $\mf{g}_{-\alpha}(F_{\alpha})$, and we set 
\begin{equation}
f(X_{\alpha}) := \frac{[X_{\alpha}, \tau_{\alpha} X_{\alpha}]}{H_{\alpha}} \in F_{\alpha}^{\times}.
\end{equation}
It is easily checked that $f(X_{\alpha})$ lies in $F_{\pm \alpha}^{\times}$, and is well-defined up to norms from $F_{\alpha}$, so if we set $\kappa_{\alpha} \co F_{\pm \alpha}^{\times} \rightarrow \{ \pm 1\}$ to be the quadratic character associated to $F_{\alpha}/F_{\pm \alpha}$ (which kills norms from $F_{\alpha}$) then 
\begin{equation}\label{eqn toral invariant}
f(\alpha) = \kappa_{\alpha} \left( \frac{[X_{\alpha}, \tau X_{\alpha}]}{H_{\alpha}} \right) \in \{ \pm 1\}
\end{equation}
is independent of the choice of $X_{\alpha}$. 

\begin{remark}\label{kappa is legendre}To flesh this out, we note that if $\alpha$ is symmetric and inertially symmetric, i.e. $F_{\alpha}/F_{\pm \alpha}$ is totally ramified quadratic (which applies for all roots  in our situation), then $\kappa_{\alpha}$ can be identified with the Legendre symbol on $F_{\pm \alpha}^{\times}/N_{F_{\alpha}/F_{\pm \alpha}} (F_{\alpha}^{\times} ) \cong k_{F_\alpha}^{\times}/k_{F_\alpha}^{\times 2}$. As explained in \cite[Chapter V, Section 3, Corollaries 5 and 7]{Serre79}, the latter isomorphism is the composition of the canonical isomorphisms \[ F_{\pm \alpha}^{\times}/N_{F_{\alpha}/F_{\pm \alpha}} (F_{\alpha}^{\times} ) \longleftarrow \cal{O}_{F_{\pm \alpha}}^{\times}/N_{F_{\alpha}/F_{\pm \alpha}} (\cal{O}_{F_{\alpha}}^{\times} ) \longrightarrow k_{F_\alpha}^{\times}/k_{F_\alpha}^{\times 2}. \]
\end{remark}

From the toral invariant $f_{(G,S)}$, we can construct a character $\epsilon_f \co S(F) \rightarrow \CC^{\times}$ as explained in \cite[\S 4.6]{Kal15}. It is determined by the formula \eqref{kaletha toral invariant} below, so we will omit the definition from first principles.

\subsubsection{Computing the toral invariant} 

We begin by recalling a useful result (\cite[Lemma 4.12]{Kal15}) for computing the toral invariant. To state it, we introduce some notation. We say that the root values of $\gamma \in S(F)$ are \emph{topologically semi-simple} (resp. \emph{unipotent}) if for all $\alpha \in R(S,G)$ the element $\alpha(\gamma) \in F_{\alpha}^{\times}$ is topologically semi-simple (resp. unipotent) (see \cite{AS08} for the terminology). 

\begin{lemma}
If the action of $I$ on $X^*(S)$ is tame and generated by a regular elliptic element, then for every $\gamma \in S(F)$ whose root values are topologically semi-simple we have 
\[
\epsilon_f(\gamma) = \prod_{\substack{\alpha \in R(S,G)_{\mrm{sym}}/\Gamma \\ \alpha(\gamma) \neq 1} } f_{(G,S)}(\alpha).
\]
For every $\gamma$ whose root values are topologically unipotent, we have $\epsilon_f(\gamma)=1$. 
\end{lemma}

The assumption is satisfied for all epipelagic parameters. This implies that $\epsilon_f$ factors through $S/S_{1/2} \cong C_{\lambda} $, and is given by
\begin{equation}\label{kaletha toral invariant}
\epsilon_f(\gamma) = \prod_{\substack{\alpha  \in R(S,G)/\Gamma \\ \alpha(\gamma) \neq 1}} f(\alpha).
\end{equation}
(Note that in our case, every root is inertially symmetric, hence a fortiori symmetric.) We have $C_{\lambda} \cong \prod_{i=1}^m \mu_2(F_i)$. Let 
\begin{equation}\label{c_i}
c_i   = (0, \ldots, 0, \underbrace{1}_i, 0, \ldots, 0) \in  \prod_{i=1}^m \mu_2(F_i) \cong C_{\lambda}
\end{equation}
be the ``indicator'' of the $i$th component. The roots which are non-trivial on $c_i$ are the $\pm \alpha^{(i,j)}_{r,s}$ where $j \neq i$. The splitting field of $\alpha^{(i,j)}_{r,s}$ is the unramified extension of $E$ of degree $d_{ij} := [d_i,d_j]$ (the least common multiple of $d_i$ and $d_j$), which we denote by $E_{d_{ij}}$.\\

To compute the toral invariant we introduce some new notation. The factorization $p_{\lambda}(x) = \prod p_i(x)$ induces a splitting of our Hermitian space $(V,\langle \cdot, \cdot \rangle)$ into eigenspaces for $\wt{\lambda}$, regarded as a unitary operator on $(V,\langle \cdot, \cdot \rangle)$, which we write as
\[
(V , \langle \cdot, \cdot \rangle) \cong \bigoplus_{i=1}^m  (E_i, \langle \cdot, \cdot \rangle_i) .
\]
We also denote by $V_i$ the underlying $E$-inner product space structure of $(E_i, \langle \cdot, \cdot \rangle_i)$.

Now comes a simple but important point. Since $\wt{\lambda}$ is unitary, the form $\langle \cdot, \cdot \rangle_i$ has the property that $\langle e x, y \rangle = \langle x, \ol{e} y \rangle$ for any $e \in E_i$, where $e \mapsto \ol{e}$ is the conjugation in $\Gal(E_i/F_i)$. By the non-degeneracy of the trace pairing, any hermitian form with this property can be written as
\begin{equation}\label{hermitian trace pairing}
\langle x, y \rangle_i = \Tr_{E_i/E} (\eta_i x \ol{y})
\end{equation}
for some $\eta_i \in F_i^{\times}$ (the hermitian property forces $\eta_i$ to be fixed by $\Gal(E_i/F_i)$). 

Since $\lambda$ was regular semisimple over $k_F$, there exists an $\Cal{O}_E$-lattice $\Lambda$ for $V$ and a compatible splitting $\Lambda = \bigoplus_{i=1}^m \Lambda_i$. Thus we have a similar story over the residue field $k_F$, which will be useful for the computation. We abuse notation by also using $\eta_i$ to denote the image of  $\eta_i$ under the isomorphism $F_i^{\times}/N_{E_i/F_i} E_i^{\times} \cong k_{F_i}^{\times}/k_{F_i}^{\times 2}$.

 Choose a basis $\{v^{(i)}_r\}_{r=1, \ldots, d_i}$ for $V_i$. Over $E_{d_{ij}}$ we pick generators for the associated root groups: for $i \neq j$, let $X^{(i,j)}_{r,s}\in \mf{g}_{\alpha^{(i,j)}_{r,s}}$ be the element of $\mf{g}(E_{d_{ij}})$ sending $v^{(i)}_r  \mapsto v^{(j)}_s$ and sending all the other basis vectors to $0$: thus $X^{(i,j)}_{r,s}$ is the ``elementary matrix'' with a single non-zero entry of $1$ in the entry corresponding to the pair $v^{(i)}_r, v^{(j)}_s$. (The roots with $i = j$ will not contribute to the present calculation.)
 
 Corresponding to the root $\alpha^{(i,j)}_{r,s}$ we have the coroot $H^{(i,j)}_{r,s} \in \mf{g}(E_{d_{ij}})$ which can be identified with the matrix sending 
 \begin{align*}
 v^{(i)}_r &  \mapsto  v^{(i)}_r\\
  v^{(j)}_s & \mapsto  - v^{(j)}_s.
 \end{align*}
We must then calculate $\tau_{\alpha^{(i,j)}_{r,s}}(X_{r,s}^{(i,j)})$. Since $\mf{g}$ consists of anti-Hermitian matrices, $\tau_{\alpha^{(i,j)}_{r,s}}(X_{r,s}^{(i,j)})$ may be identified with the negative of the adjoint of $X_{r,s}^{(i,j)}$. Since the hermitian form is given by \eqref{hermitian trace pairing}, we can choose the basis $\{v_r^{(i)}\}_{r=1, \ldots, d_i}$ so that the hermitian form is represented by the matrix
\[
\begin{pmatrix} \eta_i  & &  & \\ &  \sigma(\eta_i)&  \\ &   & \sigma^2(\eta_i) \\ & & & \ddots   \end{pmatrix}. 
\]
Therefore $\tau_{\alpha^{(i,j)}_{r,s}}(X_{r,s}^{(i,j)})$ takes $v^{(j)}_s \mapsto - \frac{\sigma^{s-1}( \eta_j)}{\sigma^{r-1}(\eta_i) }v^{(i)}_r$ and sends the other basis vectors to $0$. Thus 
\[
[X_{r,s}^{(i,j)}, \tau_{\alpha^{(i,j)}_{r,s}} X_{r,s}^{(i,j)}] = - \frac{\sigma^{s-1}( \eta_j)}{\sigma^{r-1}(\eta_i) }  H_{r,s}^{(i,j)}.
\]
Hence, by the definition of the toral invariant \eqref{eqn toral invariant} we have 
\[
f_{(G,S)}(\alpha^{(i,j)}_{r,s}) =  \kappa_{\alpha^{(i,j)}_{r,s}} \left( - \frac{\sigma^{s-1}( \eta_j)}{\sigma^{r-1}(\eta_i) } \right).	
\]
Since all the roots are inertially symmetric, this is the same as the Legendre symbol (using Remark \ref{kappa is legendre})
\[
f_{(G,S)}(\alpha^{(i,j)}_{r,s}) = \leg{-  \sigma^{s-1}( \eta_j) / \sigma^{r-1}(\eta_i) }{q^{d_{ij}}}
\]
where ${d_{ij}} = [d_i, d_j]$ is the degree of the residue field of $F_{\alpha}/F$, the symbol $\leg{\cdot}{q^{d_{ij}}}$ is the quadratic character of $\F_{q^{d_{ij}}	} $, and where we are invoking the earlier abuse of notation to view $ \sigma^{s-1}( \eta_j) / \sigma^{r-1}(\eta_i)$ in the residue field modulo squares. Since $\sigma$ is a lift of Frobenius, we have 
\begin{equation}\label{toral invariant}
 f_{(G,S)}(\alpha^{(i,j)}_{r,s})  = \leg{ - \sigma^{s-1}( \eta_j) / \sigma^{r-1}(\eta_i) }{q^{{d_{ij}}}} = \leg{- \eta_j^{q^{s-1}} / \eta_i^{q^{r-1}} }{q^{{d_{ij}}}} = \leg{-\eta_j/\eta_i}{q^{{d_{ij}}}}.
\end{equation}

Using these calculations and \eqref{kaletha toral invariant} we may finally describe the character $\epsilon_f \co S(F) \rightarrow \CC^{\times}$.  

\begin{cor}\label{cor: toral invariant 1}
Let $c_i$ be as in \eqref{c_i}. Then we have 
\[
\epsilon_f(c_i) =   \leg{-1}{q}^{d_i(n-d_i)}  \leg{\eta_i}{q^{d_i}}^{n} \prod_{j}  \leg{\eta_j}{q^{d_j}}^{d_i} .
\]
\end{cor}

\begin{proof}
According to \eqref{kaletha toral invariant} we have 
\[
\epsilon_f(c_i) = \prod f_{(G,S)}(\alpha^{(i',j')}_{r,s}) 
\]
where the product runs over Galois orbits of roots $\alpha^{(i',j')}_{r,s}$ such that $\alpha^{(i',j')}_{r,s}(c_i) \neq 1$. It is easy to see that this occurs if and only if \emph{exactly one} of $i',j'$ equals $i$. The Galois orbit of $\alpha^{(i',j')}_{r,s}$ has size $d_{i'j'} = [d_{i'},d_{j'}]$, so for each fixed $i',j'$ there are $(d_{i'},d_{j'})$ Galois orbits. Therefore by the computation \eqref{toral invariant} this product equals 
\begin{equation}\label{toral invariant eq 1}
\prod f_{(G,S)}(\alpha^{(i',j')}_{r,s})  = \prod_{j \neq i} \leg{-\eta_j/\eta_i}{q^{[d_i,d_j]}}^{(d_i,d_j)}.
\end{equation}
To simplify this expression, we use the identity $\left( \frac{ x^d }{q} \right) = \left( \frac{x}{q^d} \right)$ from Lemma \ref{Legendre norm} repeatedly in \eqref{toral invariant eq 1} to rewrite it as
\begin{align*}
\prod f_{(G,S)}(\alpha^{(i',j')}_{r,s})  &  = \prod_{j\neq i} \leg{-1}{q^{d_id_j}} \prod_{j \neq i}  \leg{\eta_i}{q^{d_id_j}} \prod_{j\neq i}  \leg{\eta_j}{q^{d_id_j}} =  \prod_{j\neq i} \leg{-1}{q}^{d_id_j} \prod_{j \neq i}  \leg{\eta_i}{q^{d_i}}^{d_j} \prod_{j\neq i}  \leg{\eta_j}{q^{d_j}}^{d_i} \\
&= \leg{-1}{q}^{d_i(n-d_i)}  \leg{\eta_i}{q^{d_i}}^{n-d_i} \prod_{j\neq i}  \leg{\eta_j}{q^{d_j}}^{d_i} =  \leg{-1}{q}^{d_i(n-d_i)}  \leg{\eta_i}{q^{d_i}}^n \prod_{j}  \leg{\eta_j}{q^{d_j}}^{d_i}.
\end{align*}

\end{proof}

In order to elucidate the dependence of the toral invariant on the stable orbit, we now manipulate Corollary \ref{cor: toral invariant 1} into another form. 

The Hermitian form on $\Lambda$ descends to a symmetric bilinear form on $\ol{V} := \Lambda/\varpi_E \Lambda$, and we have a compatible splitting $\ol{V} \cong \bigoplus_i \ol{V}_i$. Let $D_i$ be the the discriminant of $\ol{V}_i$, as defined in Appendix \ref{app: disc}. Then by Lemma \ref{disc trace pairing} we have $\leg{D_j}{q} = (-1)^{d_j-1} \cdot \leg{\eta_j}{q^{d_j}}$, hence
\[
 \prod_{j}  \leg{\eta_j}{q^{d_j}}^{d_i}  = \left( \prod_j (-1)^{d_j-1} \leg{D_j}{q} \right)^{d_i}.
 \]
Let $D$ be the discriminant of $\ol{V}$; notice that this is independent of $\lambda$. Since  $\prod_j  \leg{D_j}{q} = \leg{D}{q}$ we can rewrite Corollary \ref{cor: toral invariant 1} as follows. 

\begin{cor}\label{cor: toral invariant 2}  Let $c_i$ be as in \eqref{c_i}. 
Let $D$ be the discriminant of $\ol{V}$. Then we have 
\[
\epsilon_f(c_i) =   \leg{-1}{q}^{d_i(n-d_i)}  \cdot  \left( \prod_j (-1)^{d_j-1}  \right)^{d_i} \cdot \leg{\eta_i}{q^{d_i}}^{n} \cdot \leg{D}{q}^{d_i}
\]
\end{cor}

\begin{remark}
Since the $d_i$ and $D$ are independent of the particular rational orbit within the stable orbit of $\lambda$, the only term in the expression that depends on the rational orbit of $\lambda$ is $\leg{\eta_i}{q^{d_i}}^{n} $, which disappears when $n$ is even. 
\end{remark}

\subsection{Explication of local Langlands for tori}\label{LLC tori}

Following \S \ref{section Langlands parameters}, we now have the ``right'' pair $(S, \chi)$ to input into the local Langlands correspondence for tori, obtaining a Langlands parameter $W_F \rightarrow {}^L S$.
at which point we will need to pick the correct admissible embedding ${}^L j \co {}^L S \hookrightarrow {}^L G$. In order to do so, we will have explicate the data that comes out of the local Langlands correspondence. 

First let us flesh out the character $\chi$. It is a product $\chi_{S, {}^Lj} = \chi_{\lambda} \cdot \psi \cdot \epsilon_f^{-1}$ where $\epsilon_f$ is the character coming from the toral invariant, which we just computed in Corollary \ref{cor: toral invariant 2}. Note that both $\psi$ and $\epsilon_f$ factor through $S/S_{1/2} \cong C_{\lambda}$. On the other hand, $\chi_{\lambda}$ is the character of $S_{1/2}$ obtained as the composition 
\[
S_{1/2} \rightarrow S_{1/2}/S_1 = \mf{s}_{1/2}/\mf{s}_1 \hookrightarrow \msf{V}_x \xrightarrow{\chi \circ \lambda} \CC^{\times}.
\]

Recall the description \eqref{torus S}, $S = \prod_{i=1}^m \mrm{U}_1(E_i/F_i)$. We make some preliminary observations concerning the local Langlands correspondence for groups of the form $U_1(E_i/F_i)$. There is a surjection 
\begin{align*}
h \co E_i^{\times}& \twoheadrightarrow \mrm{U}_1(E_i/F_i) \\
x &\mapsto x \ol{x}^{-1} .
\end{align*}
By duality (contravariance of formation of $L$-groups for tori), $h$ induces an injection of dual groups $\wh{h} \co \wh{\mrm{U}}_1(E_i/F_i)  \hookrightarrow  \wh{E}_i^{\times} \cong \Ind_{W_{E_i}}^{W_F} \CC^{\times}$. The LLC for $\mrm{U}_1(E_i/F_i)$ can be embedded in the LLC for $E_i^{\times}$ via the diagram
\[
\begin{tikzcd}
\Hom(\mrm{U}_1(E_i/F_i), \CC^{\times}) \ar[r, hook, "h^*"]  \ar[dd, dotted] &   \Hom( E_i^{\times},\CC^{\times})  \ar[d, leftrightarrow, "\text{LLC}"] \\
 &  H^1(W_{E_i}, \CC^{\times}) \ar[d, equals, "\text{Shapiro's Lemma}"] \\
H^1(W_F, \wh{\mrm{U}}_1(E_i/F_i) ) \ar[r, hook, "\wh{h}"] &   H^1(W_F,  \Ind_{W_{E_i}}^{W_F} \CC^{\times})   \\
\end{tikzcd}
\]
The observation is that the dotted arrow is the Local Langlands Correspondence for $\mrm{U}_1(E_i/F_i)$. In fact this is an instance a general functoriality property for the LLC for tori, as is clear from the explicit construction of this correspondence in \cite{Lang97}.


We will use this diagram to explicate certain information which will be necessary for computing the admissible embedding. From $(S, \chi)$ we obtain homomorphisms $\varphi_i \co W_{E_i} \rightarrow \CC^{\times}$, and we will need to know what these maps do to wild inertia and certain lifts of Frobenius. According to our convention \eqref{functional duality convention}, the character $\chi_{\lambda}$ is trivial on $S_1$, and can be described on $S_{1/2} = S(F) \cap (1+\varpi_E \Cal{O}_{E_{\wt{\lambda}}}) $ as 
\[
\chi_{\lambda}(y) = \chi( \ol{ \Tr_{E_{\wt{\lambda}}/E} (\lambda \log y  ) }  )
\]
where we use $\varpi_E$ to normalize the logarithm $\log \co S(F) \cap (1+\varpi_E \Cal{O}_{E_{\wt{\lambda}}}) \xrightarrow{\sim} \mf{s}_{1/2}$ and the overline indicates reduction modulo $\varpi_E$. If $\delta_i$ is a root of $p_i(x)$ in  $k_{F_{d_i}}$, then the component corresponding to $\mrm{U}_1(E_i/F_i)$ can be written as 
 \begin{equation}
\chi_{\lambda}|_{\mrm{U}_1(E_i/F_i)}(y) =   \chi(\Tr_{k_{F_i}/k_F}({ \delta_i(\varpi_E^{-1} \log y) }))
 \end{equation}
 
The map $E_i^{\times} \rightarrow U_1(E_i/F_i)$ induces multiplication by $2$ on $U^1(E_i)/U^2(E_i)   \xrightarrow{\sim}  k_{E_i}$, where the identification is via $1+\varpi_E x  \mapsto x$. Use the uniformizer $\varpi_E$ to identify $P_{E_i} / P_{E_i}^{(2)} \cong k_{E_i}$. Since $\epsilon_f$ and $\psi$ both factor through $S/S_{1/2}$, the map $\varphi_i$ restricted to wild inertia is given by 
 \begin{align}\label{character on wild inertia}
 P_{E_i} \rightarrow P_{E_i} / P_{E_i}^{(2)} \cong  k_{E_i} &\rightarrow \CC^{\times} \\
  x &\mapsto \chi( 2 \Tr_{k_{F_i}/k_F}(\delta_i x) ).
  \end{align}

\subsection{The admissible embedding}\label{section admissible embedding}
The final step is to describe the correct admissible embeddings 
\[
{}^L j \co {}^L  S \hookrightarrow {}^L G.
\]
In \cite[\S 2.6]{LS87} it is described how to attach to \emph{$\chi$-data} a $\wh{G}$-conjugacy class of embeddings as above, which shall be reviewed shortly. Thus our problem can be rephrased as one of determining the correct $\chi$-data, which is explained in \cite[\S 5.2]{Kal15}. One of the interesting and novel aspects of the Langlands correspondence for epipelagic representations is that the $\chi$-data depends subtly on the parameter, whereas in earlier work \cite{K13} it had been independent of the admissible embedding.

\subsubsection{Background and notation}\label{chi data background}
For the sake of exposition, we explain some background on $\chi$-data and admissible embeddings. This will also give us a chance to fix some notation which we shall need anyway. All this material can be found in \cite[\S 2]{LS87}, but it may be easier on the reader to have the relevant facts collected here, presented in a manner streamlined for our current needs.

 \begin{defn}\label{defn chi datum}
 A \emph{$\chi$-datum} is a set $\{\chi_{\alpha} \co F_{\alpha}^{\times} \rightarrow \CC^{\times} \mid \alpha \in R(S,G)\}$ satisfying:
 \begin{enumerate}
 \item $\chi_{-\lambda} = \chi_{\lambda}^{-1}$,
 \item $\chi_{\sigma\alpha} = \chi_{\alpha} \circ  \mrm{conj}_{\sigma^{-1}}$ for $\sigma \in G_F$.
 \item If $[F_{\alpha}:F_{\pm\alpha}]=2$, then $\chi_{\alpha}$ extends  the quadratic character attached to $F_{\alpha}/F_{\pm\alpha}$ by local class field theory.
 \end{enumerate} 
 \end{defn}

Fix a root datum $(\wh{B},\wh{T},\{X_{\alpha}^{\vee}\})$ for $\wh{G}$. Recall that an admissible embedding is an embedding $\xi \co {}^L S \rightarrow {}^L G$ such that 
\begin{enumerate}
\item $\xi$ maps $\wh{S} $ isomorphically to $\wh{T}$,
\item $\xi(w) \in \wh{G} \times w$. 
\end{enumerate}
Thus composing a Langlands parameter $W_F \rightarrow {}^L S$ with an admissible embedding $\xi \co {}^L S \hookrightarrow {}^LG$ produces a Langlands parameter into ${}^L G$. A $\chi$-datum for $S$ can be used to parametrize the $\wh{G}$-conjugacy classes of admissible embeddings $^{L}j \co {}^L S \rightarrow {}^LG$, as we now explain. 

Since any admissible embedding $\xi \co {}^L S \hookrightarrow {}^L G$ is already specified on $\wh{S} \rtimes 1 \subset {}^L S = \wh{S} \rtimes W_F$, it is determined by its restriction to $1 \rtimes W_F \subset {}^L S$. Since the image of $1 \rtimes W_F$ must normalize $\xi(\wh{S}) = \wh{T}$, we have for any $w \in W_F$, that $\xi(w) \in {}^L G$ is of the form 
\[
\xi(w) =\xi_0(w) \times w \in \wh{G} \rtimes W_F
\]
where $\xi_0(w) \in N(\wh{T}, \wh{G})$, and conjugation by $\xi(w)$ acts on $\wh{T}$ in the same way as the restriction of the Galois action via $W_F \rightarrow \Gamma_F$. The latter condition specifies the image of $\xi_0(w)$ in the Weyl group of $\wh{G}$ with respect to $\wh{T}$, which we denote $\Omega(\wh{T}, \wh{G})$. Let us denote this image of $\xi_0(w)$ in $\Omega(\wh{T}, \wh{G})$ by $\omega(w)$.

For a simple root $\alpha \in R(S,G)$, let $n(\alpha) = \exp(X_{\alpha}) \exp(-X_{-\alpha})\exp(X_{\alpha})$ denote the associated reflection in $\wh{G}$, or equivalently the image of 
$\begin{pmatrix} 0  & 1  \\ -1  & 0 \end{pmatrix}$ under the map $\SL_2 \rightarrow \wh{G}$ associated with regarding $\alpha$ as a coroot of $\wh{G}$. For $s_{\alpha}$ the simple reflection in $\Omega(\wh{T}, \wh{G})$ we denote $n(s_{\alpha}) = n(\alpha)$, and more generally for any $\omega \in \Omega(\wh{T}, \wh{G})$, we choose a reduced expression $\omega = \omega_{\alpha_1} \omega_{\alpha_2}  \ldots \omega_{\alpha_r} $ for $\omega$ as a product of simple reflections, and set 
\[
n(\omega)  =n(\alpha_1)n(\alpha_2) \ldots n(\alpha_r)\in N(\wh{T}, \wh{G}).
\] 
(This is independent of the reduced expression, according to \cite[p. 228]{LS87}.) This provides a \emph{set-theoretic section} $\Omega(\wh{T}, \wh{G}) \rightarrow N(\wh{T}, \wh{G})$, and can be viewed as a candidate admissible embedding ${}^L S  \hookrightarrow {}^L G$, sending $w  \mapsto n(\omega(w)) \rtimes w$. The problem is that this is not (necessarily) a \emph{homomorphism}. To make it into a homomorphism, we need to modify the elements $n(\omega(w))$ by elements of $\wh{T}$. This amounts to splitting a certain cocycle, and the $\chi$-data provide such a splitting. 

 For $\theta = \alpha \rtimes w \in \Omega(\wh{T}, \wh{G}) \rtimes W_F$, set $n(\theta) := n(\alpha) \rtimes w$. For $\theta_1, \theta_2 \in  \Omega(\wh{T}, \wh{G}) \rtimes W_F$, we have 
 \[
 n(\theta_1) n(\theta_2) = t(\theta_1, \theta_2) n(\theta_1 \theta_2) 
 \]
where $t(\theta_1, \theta_2) \in \wh{T}$ because the actions of $ n(\theta_1) n(\theta_2)$ and $ n(\theta_1 \theta_2)$ on $\wh{T}$ are equal. Then $t(\theta_1, \theta_2)$ defines a 2-cocycle on $\Omega(\wh{T},\wh{G})$ valued in $\wh{T}$. The point is that its inflation to $W_F$ is \emph{split}. A $\chi$-data furnishes a choice of splitting $r(w)$, so that $\xi(w) = r(w) n(\omega(w)) \times w$
defines an admissible homomorphism $\xi \co {}^L T \rightarrow {}^L G$. 

In order to explain this, we unfortunately have to introduce yet more terminology. Recall that a \emph{gauge} is a function $p \co R(S, G) \rightarrow \{\pm 1\}$ such that $p(-\lambda) = - p(\lambda)$. A choice of positive system of roots induces a gauge, namely the one assigning $+1$ to the positive roots, but not all gauges arise from such a choice. We can think of a gauge as a generalization of a choice of positive system. 
 
We now summarize some material from \cite{LS87} which is useful for having a general picture of what is going on, but whose rather technical details play no role here. In \cite[Lemma 2.1]{LS87} a formulate for $t(\theta_1,\theta_2)$ is obtained, and serves as motivation to define a more general 2-cocycle $t_p(\theta_1, \theta_2)$ depending on a gauge $p$, which when $p$ is specialized to the gauge associated to the positive root system associated to the based root datum of $\wh{G}$, recovers  $t(\theta_1, \theta_2)$. In \cite[Lemma 2.1.C]{LS87} it is shown that the cohomology class of the 2-cocycle $t_p(\theta_1, \theta_2)$ is independent of the choice of gauge $p$, the point being that one can use a more convenient gauge to calculate a splitting. 

Next we describe a particular splitting $r_p$ for $t_p(\theta_1, \theta_2)$, for a convenient choice of gauge $p$ \cite[\S 2.5]{LS87}. The first step is to make certain choices for coset representatives. Given a $\chi$-datum $\{ \chi_{\alpha} \co F_{\alpha}^{\times} \rightarrow \CC^{\times} \co \alpha \in R(S,G)\}$, we use local class field theory to view the characters $\chi_{\alpha}$ as characters on $W_{\alpha} := W_{F_{\alpha}}$.  Let $\epsilon \in  \Aut(R(S,G))$ be the automorphism acting by $-1$ on the roots. We initially consider the case where $\Sigma := \langle \Gamma,\epsilon \rangle$ acts transitively on the roots. Fix $\alpha \in R(S,G)$ and choose representatives $\sigma_1,\ldots,\sigma_n$ for $\Gamma_{\pm \alpha} \backslash \Gamma$. The roots are then of the form $\pm \sigma_1^{-1} \alpha, \ldots, \pm \sigma_n^{-1} \alpha$. We define a gauge $p$ by $p(\alpha')=1$ if $\alpha' = \sigma_i^{-1} \alpha$ for some $i$ (i.e. appears with a positive sign). Choose $w_i \in W_F$ mapping to $\sigma_i \in \Gamma$. Then define $u_i(w) \in W_{\pm \alpha}$ by 
\begin{equation}\label{u_i formula}
w_iw  = u_i(w) w_j.
\end{equation}
Choose $v_0 \in W_{\alpha}$ and $v_1 \in  W_{\pm \alpha }  - W_{\alpha} $ if $[F_\alpha : F_{\pm \alpha}]=2$ (otherwise we just need $v_0$). For $u\in W_{\pm \alpha }$  we define $v_0(u) \in W_\alpha$ by 
\begin{equation}\label{v_0 formula}
v_0 u = v_0(u) v_{i'}
\end{equation}
where $i'=0$ or $1$ as appropriate. We will write down a function $r_p(w)$ whose coboundary is $t_p$. Still in the case where $\Sigma$ acts transitively, we define
\begin{equation}\label{r_p formula}
r_p(w) = \prod_{i=1}^n \chi_{\alpha}(v_0(u_i(w))) \otimes \sigma_i^{-1} \alpha \in \CC^{\times} \otimes X^*(T) .
\end{equation}
This is a 1-cocycle on $W_F$ valued in $\CC^{\times} \otimes X^*(T) = \wh{T}(\CC)$. 

In the general case where $\Sigma$ need not act transitively, we define a factor $r_p^{(\Cal{O})}$ for each $\Sigma$-orbit $\Cal{O}$ as above, and then set $r_p := \prod_{\Cal{O}} r_p^{(\Cal{O})}$
with each $ r_p^{(\Cal{O})}$ defined as in the transitive case.

\begin{lemma}[{\cite[Lemma 2.5.A]{LS87}}] The coboundary of $r_p$ is $t_p$. 
\end{lemma}

To summarize, the corresponding admissible embedding is $\xi  \co {}^L S  \hookrightarrow {}^L G$ sending $w  \mapsto  r_p(w)n(\omega(w)) \rtimes w$, where $r_p$ is as above.

\subsubsection{The $\chi$-data of epipelagic parameters}

We now describe Kaletha's prescription for extracting the $\chi$-data associated to an epipelagic parameter  $\varphi \co W_F \rightarrow {}^L G$ \cite[\S 5.2]{Kal15}. Given $\varphi$, we need to prescribe the characters $\chi_{\alpha} \co F_{\alpha}^\times \rightarrow \CC^{\times}$ satisfying the conditions in Definition \ref{defn chi datum}, the most important of which is that $\chi_{\alpha}$ be trivial on $N_{F_{\alpha}/F_{\pm \alpha}}(F_{\alpha}^{\times})$. Obviously $\chi_{\alpha}$ can only be non-trivial if $\alpha$ is symmetric, i.e. $F_{\alpha} \neq  F_{\pm \alpha}$, so we restrict our attention to symmetric $\alpha$.

  If $\alpha$ is symmetric but inertially asymmetric, then there exists a unique unramified character satisfying the desired conditions of $\chi$-datum, which is what one takes for $\chi_{\alpha}$. However, in our case of interest all roots are inertially symmetric, so this will never apply. 
  
  If $\alpha$ is inertially symmetric, then there are exactly two tamely ramified characters that satisfy the conditions of Definition \ref{defn chi datum}, and we need to use the information encoded in $\varphi$ to specify the right one \cite[p.40-41]{Kal15}. It is enough to specify the character on a uniformizer $\varpi  \in F_{\alpha}^{\times}$, since the collection of all uniformizers generate the multiplicative group. Restricting $\varphi$ to the wild inertia subgroup $P_F$, and composing with the root $\alpha$ of $\wh{T}$, we have a homomorphism 
  \[
  P_F \hookrightarrow W_F \xrightarrow{\varphi} \wh{T} \xrightarrow{\alpha} \CC^{\times}.
  \]
Viewing $P_F \cong P_{F_{\alpha}} \subset W_{F_{\alpha}}$, this composite extends to $W_{F_{\alpha}}$, hence induces by local class field theory a character of the 1-unit group $
 \xi_{\alpha} \co U_{F_{\alpha}}^1 \rightarrow \CC^{\times}$. By assumption this homomorphism is trivial on $U_{F_{\alpha}}^2$. Using the choice of uniformizer $\omega$ we obtain a character
 \begin{equation}\label{character depending on uniformizer}
 \xi_{\alpha, \omega} \co k_{F_{\alpha}} \xrightarrow{x \mapsto \omega x + 1 } U_{F_{\alpha}}^1 / U_{F_{\alpha}}^2 \xrightarrow{\xi_{\alpha}} \CC^{\times}.
 \end{equation}
 Then we set 
 \begin{equation}\label{epipelagic chi data} 
 \chi_{\alpha}(\omega) = \lambda_{F_{\alpha}/F_{\pm \alpha}}(\xi_{\alpha, \omega})^{-1}
 \end{equation}
 where $\lambda$ is the Langlands $\lambda$-function of \cite[Theorem 2.1]{Lang76}. In the case at hand, namely that of a tamely ramified quadratic extension, there is a concrete description of $\lambda$ as a normalized Gauss sum: if $q := \# k_{F_{\pm \alpha}}$, then \cite[Lemma 1.5]{BH05}
 \begin{equation}\label{lambda function}
\lambda_{F_{\alpha}/F_{\pm \alpha}}(\xi_{\alpha, \omega}) = q^{-1/2} \sum_{x \in k_{F_{\pm \alpha}}^{\times}} \leg{x}{q} \xi_{\alpha, \omega}(x).
 \end{equation}
 Note that this a (fourth) root of unity.

 \subsubsection{Computation of the admissible embedding} 	\label{section compute admissible embedding}
We now undertake the task of ``computing'' the admissible embedding 
\[
{}^L j \co {}^L S \rightarrow {}^L G.
\]
We begin with some general observations. Each ``anti-coboundary'' of $t_p$ is tautologically a 1-cochain with coboundary $r_p$, so the $\wh{G}$-conjugacy classes of such splittings is a torsor for $H^1(W_F,\wh{S})$. Since by Shapiro's Lemma we have $H^1(W_F, \wh{S}) =  \bigoplus_i H^1(W_{F_i}, \wh{\mrm{U}}_1(E_i/F_i))$, it suffices to specify a system of classes in $H^1(W_{F_i}, \wh{\mrm{U}}_1(E_i/F_i))$ for each $i$. By the Local Langlands Correspondence for tori, the datum of a cohomology class in $H^1(W_{F_i}, \wh{\mrm{U}}_1(E_i/F_i))$ is equivalent to that of a character 
\begin{equation}\label{eqn: phi_i}
\phi_i \co \mrm{U}_1(E_i/F_i) = \{ x \in E_i^{\times} \mid N_{E_i/F_i}(x) = 1\} \rightarrow \CC^{\times}.
\end{equation}
In fact, since by construction \cite[\S 5.2]{Kal15} the $L$-embedding is made with \emph{tamely ramified} $\chi$-data, each such character factors through the prime-to-$p$ quotient of $\mrm{U}_1(E_i/F_i)$, which is just $\{ \pm 1\}$.

Since $\wh{\mrm{U}}_1(E_i/F_i) \cong \CC^{\times}$ as a group, with the Galois action factoring (non-trivially) through $\Gal(E_i/F_i) \cong \Z/2\Z$, the restriction map $ H^1(W_{F_i}, \wh{\mrm{U}}_1(E_i/F_i))  \hookrightarrow H^1(W_{E_i}, \CC^{\times}) $ is injective, so it suffices to describe the image in $H^1(W_{E_i}, \CC^{\times})  $ for each $i$. On the other side of the Local Langlands Correspondence, this corresponds to inflating the character via the map $E_i^{\times} \rightarrow  \mrm{U}_1(E_i/F_i)$ given by $x \mapsto x\ol{x}^{-1}$ (as we discussed in \S \ref{LLC tori}). It suffices to compute the value of the inflated character on $\varpi_{E}$, since $\ol{\varpi_E } = - \varpi_E$ so that $\varpi_E \ol{\varpi_E}^{-1}$ generates the prime-to-$p$ quotient of $\mrm{U}^1(E_i/F_i)(F_i)$.
 
We now undertake one last simplification. Let $\sigma_E = \Art_E(\varpi_E)$, so $\sigma_E$ is a lift of the (geometric) Frobenius on $k_E$. Recall the Verlagerung functoriality of local class field theory: 
\[
\begin{tikzcd}
E_i^{\times} \ar[r, "\Art_{E_i}"] & W_{E_i}^{\ab}  \\
E^{\times} \ar[u, hook] \ar[r, "\Art_E"]&  W_E^{\ab}  \ar[u, "\Ver_{E_i/E}"']
\end{tikzcd}
\]
This implies at $\Art_{E_i}(\varpi_E) = \Ver_{E_i/E}(\sigma_E)$. To compute $\Ver_{E_i/E}(\sigma_E)$ we note that we may take $1, \sigma_E,  \sigma_E^2, \ldots,\sigma_E^{d_i-1}$ as coset representatives for $W_E/W_{E_i} \cong \Z/d_i \Z$. Using these representatives, it is trivial to calculate that $\Ver(\sigma_E)= \sigma_E^{d_i}$. The upshot is that, if we view the inflation of  $\phi_i$ to $E_i^{\times}$ as a cocycle in $H^1(W_{E_i},\CC^{\times})$ via local class field theory, then the embedding ${}^L j$ will be determined by computing $\phi_i$ on $\sigma_E^{d_i}$. \\

Now we finally make the embedding explicit. We need to compute the $\chi$-datum $\{\chi_{\alpha}\co \alpha \in R(S,G)\}$ and then the admissible embedding $r_p(\sigma_E^{d_i})$. The first task is to calculate each factor $r_p$ from \eqref{r_p formula}. For this we have to organize the roots into Galois orbits. Recall that the roots were denoted $\alpha^{(i,j)}_{r,s}$. We divide into cases according to whether or not $i = j$. \\

\noindent \textbf{Case 1: $i \neq j$.} As we already observed in \S \ref{toral invariant}, there are $(d_i, d_j)$ Galois orbits of roots of the form $\alpha_{r,s}^{(i,j)}$, and the size of each orbit is $[d_i,d_j]$, so 
we have $F_{\alpha^{(i,j)}_{r,s}} = E_{[d_i,d_j]}$, the unramified extension of $E$ of degree $[d_i,d_j]$. Since the conjugation of $E/F$ acts the roots as negation, we have $F_{ \pm  \alpha_{r,s}^{(i,j)}} = F_{[d_i, d_j]}$. Thus we have   
\[
W_{\pm \alpha_{r,s}^{(i,j)}} \backslash W \cong \Gal(F_{[d_i, d_j]}/F)\cong \Z/[d_i, d_j]\Z.
\]
Let $\Cal{O}^{(i,j)}_{r,s}$ denote the orbit of $\alpha^{(i,j)}_{r,s}$. 

We must now choose coset representatives. We choose representatives $w_1, \ldots, w_{[d_i, d_j]}$ for $W_{\pm \alpha_{r,s}^{(i,j)}} \backslash W$ to be the powers of the lift of Frobenius, say $w_i = \sigma_E^{i-1}$. We then choose coset representatives $v_0 = \Id$ and $v_1$ arbitrary for $W_{\alpha_{r,s}^{(i,j)}} \backslash W_{\pm \alpha_{r,s}^{(i,j)}}$. 

We now compute using \eqref{u_i formula} and \eqref{v_0 formula}. First applying $\eqref{u_i formula}$ to $\sigma_E^a$ for $a<[d_i,d_j]$, we see that $u_i(\sigma_E^a)$ is determined by $\sigma_E^{i-1} \sigma_E^a = u_i(\sigma_E^a) \sigma_E^{a+i-1\mod{[d_i,d_j]}}$, so that 
\[
u_i(\sigma_E^a) = \begin{cases} 1 & i \leq [d_i,d_j]-a, \\ \sigma_E^{[d_i, d_j]} & \text{otherwise.} \end{cases} 
\]
Since each $u_i(\sigma_E^a)$ already lies in $W_{\alpha}$, we have $v_0(u_i(w)) = u_i(w)$. So for this $w$ we have, according to \eqref{r_p formula}
\[
r_p^{\Cal{O}^{(i,j)}_{r,s}}(\sigma_E^{d_i}) = \prod_{t=1}^{d_i} \chi_{\alpha^{(i,j)}_{r,s}}(\sigma_E^{[d_i,d_j]}) \otimes \alpha^{(i,j)}_{r+t, s+t}.
\]
Note that $r$ is valued in $\Z/d_i \Z$ and $s$ is valued in $\Z/d_j \Z$. So as $t$ runs from $1$ to $d_i$, $r+t$ takes on every value in $\Z/d_i \Z$ exactly once. Write $\pi_i \co \wh{S} \rightarrow \wh{S}_i := \wh{\mrm{U}}_1(E_i/F_i)$ for the projection onto the $i$th component. Recalling that $\alpha^{(i,j)}_{r,s} = \alpha^{(i,j)}_{r,s} = \lambda^{(i)}_r - \lambda^{(j)}_s$, the projection of this cocycle to $\wh{S}_i$ via $\pi_i$ is
\begin{equation}\label{eq horrible1}
\pi_i(r_p^{\Cal{O}^{(i,j)}_{r,s}}(\sigma_E^{d_i}))  = \chi_{\alpha^{(i,j)}_{r,s}}(\sigma_E^{[d_i,d_j]}) \otimes \sum_{r=1}^{d_i} \lambda^{(i)}_r.
\end{equation}
We write $\Delta_{i} := \sum_{r=1}^{d_i} \lambda^{(i)}_r \in X^*(S_i)$, since this is the cocharacter corresponding to the  ``diagonal'' embedding in $\wh{S}_i \cong (\CC^{\times})^{d_i}$. Then we rewrite \eqref{eq horrible1} as 
\[
\pi_i(r_p^{\Cal{O}^{(i,j)}_{r,s}}(\sigma_E^{d_i}))  = \Delta_i(\chi_{\alpha^{(i,j)}_{r,s}}(\sigma_E^{[d_i,d_j]}) ) \in \wh{S}_i(\CC^{\times}).
\]
Now it only remains to compute $\chi_{\alpha^{(i,j)}_{r,s}}(\sigma_E^{[d_i,d_j]})$. For ease of notation, we abbreviate $\alpha : = \alpha^{(i,j)}_{r,s}$ for the rest of this computation. We also set $k_d :=  k_{E_{d_{ij}}}= k_{F_d}$ for the extension of $ k := k_F$ of degree $d$. By \eqref{epipelagic chi data} we should define $\chi_{\alpha}(\omega) = \lambda_{F_{\alpha}/F_{\pm \alpha}} (\xi_{\alpha, \omega})^{-1}$, where $\xi_{\alpha, \omega} \co k_{[d_i,d_j]} \rightarrow \CC^{\times}$ is as in \eqref{character depending on uniformizer}. This will take some painful work to unravel. By the same Verlagerung computation as before, $\sigma_E^{[d_i,d_j]}$ corresponds to $\varpi_E$ under the local Artin map for $E_{[d_i,d_j]}$. To compute $\xi_{\alpha, \varpi_E}$, we refer back to the diagram from \S \ref{LLC tori}. 
\[
\begin{tikzcd}
\Hom(S(F), \CC^{\times}) \ar[dd, dotted] & \prod_i \Hom(\mrm{U}_1(E_i/F_i), \CC^{\times}) \ar[r, hook] \ar[l, equals]  \ar[dd, leftrightarrow, "\text{LLC}"] &  \prod_i \Hom( E_i^{\times},\CC^{\times})  \ar[d, leftrightarrow, "\text{LLC}"] \\
 & & \prod_i H^1(W_{E_i}, \CC^{\times}) \ar[d, equals] \\
H^1(W_F, \wh{S}) & \prod_i H^1(W_F, \wh{\mrm{U}}_1(E_i/F_i) ) \ar[r, hook] \ar[l, equals]&  \prod_i  H^1(W_F,  \Ind_{W_{E_i}}^{W_F} \CC^{\times})   \\
\end{tikzcd}
\]
The character of $S(F)$ gives, tracing through the diagram, an element of $H^1(W_F,  \prod_i \wh{\mrm{U}}_1(E_i/F_i))$. The torus $\prod_i \wh{\mrm{U}}_1(E_i/F_i) $ is embedded as $\wh{S} \xrightarrow{\sim} \wh{T} \subset \GL_n$ in the eventual Langlands correspondence, and according to \eqref{character depending on uniformizer} we need to understand the composition $P_F \hookrightarrow W_F \rightarrow \wh{S} \xrightarrow{\alpha} \CC^{\times}$.

Note that the identification (Shapiro's Lemma) 
\[
H^1(W_F,  \Ind_{W_{E_i}}^{W_F} \CC^{\times})  \xrightarrow{\sim} H^1(W_{E_i}, \CC^{\times}) 
\]
is given by restriction to $W_{E_i}$, and then evaluation of $f \in \Ind_{W_{E_i}}^{W_F} \CC^{\times} = \{ f \co W_F \rightarrow \CC^{\times} \co \ldots\}$ on the identity. In these terms, the restriction of $\alpha$ to $\Ind_{W_{E_i}}^{W_F} \CC^{\times}$ is given by the map $ \Ind_{W_{E_i}}^{W_F} \CC^{\times} \rightarrow \CC^{\times}$ evaluating on $\sigma_E^r$. In other words, we have a commutative diagram 
\[
\begin{tikzcd}
H^1(W_{E_i}, \CC^{\times})  \ar[r, "\Res"] &   H^1(P_{E_i}, \CC^{\times} ) \ar[r, "\mrm{conj}_{\sigma_E^{-r}}^*"] & H^1(P_{E_i}, \CC^{\times} )  \\
H^1(W_F,  \Ind_{W_{E_i}}^{W_F} \CC^{\times}) \ar[u, equals] \ar[r, "\Res"]  & H^1(P_F, \Ind_{W_{E_i}}^{W_F} \CC^{\times}) \ar[r, "\alpha"] & H^1(P_F, \CC^{\times}) \ar[u,"\Res"]
\end{tikzcd}
\]
This shows that the character in $H^1(P_F, \CC^{\times}) = \Hom(P_F, \CC^{\times})$ corresponding to $(S, \chi)$ is such that its restriction to $P_{E_i}$ gives the character in $H^1(P_{E_i}, \CC^{\times} )$ that we determined in \eqref{character on wild inertia}, pre-composed with conjugation by $\sigma_E^{-r}$. By the description in \eqref{character on wild inertia}, we conclude that this restriction is 
 \begin{equation}
x \mapsto   \chi(\Tr_{k_{F_{i}}/k_F}\ol{( \delta_i  x^{q^{-r}}) })^2= \chi(\Tr_{k_{F_{i}}/k_F}\ol{( \delta_i ^{q^{r}} x) })^2.
 \end{equation}

We still have not determined $\xi_{\alpha, \varpi_E}$. It is the character on $k_{F_{\alpha}}  \xrightarrow{\sim} U^1_{F_\alpha}/U^2_{F_{\alpha}}$, identified via $x  \mapsto 1+\varpi_E x$, corresponding to 
\[
\begin{tikzcd}
U^1_{F_\alpha} \ar[r, "\Art_{F_{\alpha}}"] \ar[d, "\Nm_{F_{\alpha}/F}"'] & P_{F_{\alpha}}  \ar[d, "\sim"] \\
U^1_{F} \ar[r] &  P_F \ar[r]&  \CC^{\times}
 \end{tikzcd}
\]
Note that under the identification above, the norm map $U^1_{F_\alpha}/U^2_{F_{\alpha}} \rightarrow U^1_F/U^2_F$ corresponds to $\tr \co k_{F_{\alpha}} \rightarrow k_F$. Furthermore, since we have identified the restriction of the character to $P_{E_i}$ and $E_i \subset F_{\alpha}$, we use this to see that $\xi_{\alpha, \varpi_E}$ is given by 
\[
\xi_{\alpha, \varpi_E}(x) = \frac{\chi(\Tr_{k_{d_i}/k} (\delta_i^{q^{r-1}} \Tr_{k_{[d_i,d_j]}/k_{d_i}} x))^2}{ \chi(\Tr_{k_{d_j}/k}(\delta_j^{q^{s-1}} \Tr_{k_{[d_i,d_j]}/k_{d_j}} x))^2} = \chi(\Tr_{k_{[d_i,d_j]}/k} ([\delta_i^{q^{r-1}}-\delta_j^{q^{s-1}} ]x))^2.
\]

Sadly we are not done yet: we still need to compute $\chi_{\alpha}(\varpi_E) = \lambda_{F_{\alpha}/F_{\pm \alpha}} (\xi_{\alpha, \varpi_E})^{-1}$. By \eqref{lambda function} we have 
\begin{equation}\label{lambda value 1}
\lambda_{F_{\alpha}/F_{\pm \alpha}} (\xi_{\alpha, \omega})^{-1} = \mrm{arg} \left( \sum_{x \in k_{[d_i,d_j]}^{\times}} \leg{x}{q^{[d_i,d_j]}} \chi (\Tr_{k_{[d_i,d_j]}/k} ([\delta_i^{q^{r-1}}-\delta_j^{q^{s-1	}} ] x))^2 \right)^{-1}
\end{equation}
where for $z\in \CC^{\times}$ we write $\arg(z) = \frac{z}{||z||} \in S^1$.

We simplify this terrifying expression slightly using the Hasse-Davenport relation: 

\begin{lemma}[Hasse-Davenport, {\cite[p.158-162]{IR90}}] Let $d \geq 1$ and $\chi$ be an additive character of $\F_q$. Then 
\[
- \sum_{\F_{q^d}} \leg{x}{q^d} \chi(\Tr_{\F_{q^d}/\F_q} x) = \left( - \sum_{x \in \F_q} \leg{x}{q} \chi(x) \right)^d.
\]
\end{lemma}

Applying this to \eqref{lambda value 1}, we finally obtain 
\begin{equation}\label{eqn case 1}
\chi_{\alpha}(\varpi_E) = (-1)^{[d_i,d_j]-1}  \leg{\delta_i^{q^{r-1}}-\delta_j^{q^{s-1}}}{q^{[d_i,d_j]}} \left(  \text{arg}  \sum_{x \in k} \leg{x}{q} \chi(x)^2 \right)^{-[d_i,d_j]}.
\end{equation}
Recall that there are $(d_i,d_j)$ such orbits, corresponding to $r-s = 1, \ldots, (d_i,d_j)$.\\

To summarize: the contribution from the roots of Case 1, namely those $\alpha^{(i,j)}_{r,s}$ with $i \neq j$, to $\pi_i(r_p(\sigma_E^{d_i}))$ is
\begin{equation}\label{eq: case 1 summary}
\Delta_i \otimes    \prod_{r-s = 1 }^{(d_i,d_j)} \prod_{i \neq j}  \left( (-1)^{[d_i,d_j]-1}  \leg{\delta_i^{q^{r-1}}-\delta_j^{q^{s-1}}}{q^{[d_i,d_j]}} \left( \text{arg }  \sum_{x \in k} \leg{x}{q} \chi(x)^2 \right)^{-[d_i,d_j]} \right) .
\end{equation}
We have finally finished Case 1. The exhausted reader may take comfort in the fact that the second case is significantly simpler.\\

\noindent \textbf{Case 2}. We consider roots of the form $\alpha^{(i,i)}_{r,s} = \lambda^{(i)}_r-\lambda^{(i)}_s$.  The Galois action factors through 
\[
\Gal(E_{d_i}/F) = \langle \tau \rangle \rtimes \{ \sigma_E^t \}_{t=1, \ldots, d_i}.
\]
with action given by 
\begin{align*}
\tau(\alpha^{(i,i)}_{r,s}) &= - \alpha^{(i,i)}_{r,s} \\
\sigma_E(\alpha^{(i,i)}_{r,s}) &= \alpha^{(i,i)}_{r+1,s+1}
\end{align*} 
where the subscripts are always considered modulo $d_i$. 

The orbit of $\alpha^{(i,i)}_{r,s}$ under Frobenius never meets $-\alpha^{(i,i)}_{r,s}$ \emph{unless} $r-s \equiv d_i/2 \mod{d_i}$ (implicitly forcing $d_i$ to be even). So this breaks us into two subcases. \\

\noindent \textbf{Case 2(a): $r-s \not\equiv d_i/2 \mod{d_i}$.} Arguing as above, we find that 
\begin{align*}
r_p^{\Cal{O}^{(i,i)}_{r,s}}(\sigma_E^{d_i}) &= \prod_{t=1}^{d_i} \chi_{\alpha^{(i,i)}_{r,s}}(\sigma_E^{d_i}) \otimes \alpha^{(i,i)}_{r-t, s-t} = \chi_{\alpha^{(i,i)}_{r,s}}(\sigma_E^{d_i}) \otimes \sum_{t=1}^{d-i} \left(\lambda^{(i)}_{r+t} - \lambda^{(i)}_{s+t} \right). 
\end{align*}
As $t$ runs from $1$ to $d_i$, both $r-t$ and $s-t$ assume every value mod $\Z/d_i \Z$ exactly once, so that the last sum cancels out to $0$. Therefore, this case contributes trivially to $r_p$. \\

\noindent \textbf{Case 2(b): $i=j$, $r-s \equiv d_i/2 \mod{d_i}$.} In this case we have $\sigma_E^{d_i/2} \alpha^{(i,i)}_{r,s} = -\alpha^{(i,i)}_{r,s}$. Abbreviating $\alpha := \alpha^{(i,i)}_{r,s}$, we find that $F_{\pm \alpha}  = F_{d_i/2}$ while $F_{\alpha} \subset E_{d_i}$ is the fixed field of $\tau \circ \sigma_E^{d_i/2}$, which is the quadratic ramified extension $E_{d_i/2}'$ of $F_{d_i/2}$ distinct from $E_{d_i/2}$. 

We now proceed in the usual manner to compute $r_p$. We begin by picking cosets $\{w_i := \sigma_E^{i-1}\}_{i=1 , \ldots,  d_i/2}$ for $W_{\pm \alpha} \backslash W_F$. Then we take $v_0 = 1$ and arbitrary $v_1$ for representatives of $W_{\alpha} \backslash W_{\pm \alpha}$. We find that $v_0(u_i(\sigma_E^{d_i})) = \sigma_E^{d_i}$. 

We must then determine
\begin{align*}
r_p^{\Cal{O}} (\sigma_E^{d_i}) = \prod_{t=1}^{d_i/2} \chi_{\alpha}(\sigma_E^{d_i}) \otimes \alpha^{(i,i)}_{r+t, s+t} = \chi_{\alpha}(\sigma_E^{d_i}) \otimes \sum_{t=1}^{d_i/2}  \left( \lambda^{(i)}_{r+t} - \lambda^{(i)}_{r-d_i/2+t} \right) = \Delta_i(\chi_{\alpha}(\sigma_E^{d_i}))
\end{align*}
where the third equality uses the fact that $\chi_{\alpha}(\sigma_E^{d_i}) = \pm 1$ is equal to its inverse. It remains to compute $\chi_{\alpha}(\sigma_E^{d_i})$. Using norm functoriality for local class field theory,
\[
\begin{tikzcd}
F_{\alpha}^{\times} \ar[r, "\Art_{F_{\alpha}}"] \ar[d, "\Nm_{F_{\alpha}/F}"'] & W_{F_{\alpha}}^{\ab} \ar[d] \\
F^{\times} \ar[r, "\Art_F"]  & W_F^{\ab} 
\end{tikzcd}
\]
we see that $\chi_{\alpha}(\sigma_E^{d_i}) = \pm 1$ is the value of the character $\chi_{\alpha}$ from the $\chi$-datum on an element of $F_{\alpha}^{\times}$ whose norm down to $F$ coincides with $N_{E/F}(\varpi_E^{d_i}) = \varpi_E^{2d_i}$. A convenient such choice is $-\varpi_E^2$, which even lies in $F$ since we arranged that $\Nm_{E/F} (\varpi_E) = -\varpi_E^2$ (cf. \S \ref{sec: notation}). This $-\varpi_E^2$ is a uniformizer of $F_{d_i/2}$, which is the norm of a uniformizer (namely $\varpi_E$) from the ramified quadratic extension $E_{d_i/2}/F_{d_i/2}$, so it cannot be a norm from $E_{d_i/2}'$ to $F_{d_i/2}$. Since $\chi_{\alpha}$ always extends the quadratic character on $F_{\pm \alpha}^{\times}$ corresponding by local field theory to $F_{\alpha}/F_{\pm \alpha}$, this shows that $\chi_{\alpha}(\sigma_E^{d_i}) = -1$.  

In summary, the contribution of the roots from Case 2 to $r_p(\sigma_E^{d_i})$ is simply $-1$ if $d_i$ is even and $1$ if $d_i$ is odd, which we can write uniformly  as $(-1)^{d_i-1}$. \\

Finally, putting together the computations from the two cases  (cf. \eqref{eqn case 1}) we find that 
\begin{align*}
\pi_i(r_p(\sigma_E^{d_i}) ) &= (-1)^{d_i-1} \prod_{r-s = 1 }^{(d_i,d_j)} \prod_{i \neq j}  \left( (-1)^{[d_i,d_j]-1}  \leg{\delta_i^{q^{r-1}}-\delta_j^{q^{s-1}}}{q^{[d_i,d_j]}} \left( \text{arg }  \sum_{x \in k} \leg{x}{q} \chi(x)^2 \right)^{-[d_i,d_j]} \right)   \\
\end{align*}
which we can simplify slightly to
\begin{equation}\label{r_p computed}
\pi_i(r_p(\sigma_E^{d_i}))= (-1)^{d_i-1}  \left( \text{arg}  \sum_{x \in k} \leg{x}{q} \chi(x)^2 \right)^{-d_i(n-d_i)}  \prod_{i \neq j} (-1)^{d_id_j-(d_i,d_j)}  \prod_{t=1}^{(d_i,d_j)} \leg{\delta_i - \delta_j^{q^{t-1}}}
{q^{[d_i,d_j]}} 
\end{equation}

\subsection{Assembly of Langlands parameters}

 In this section we collect the raw material from the computations to describe the Langlands parameter attached to $\rho_{\lambda, \psi}$. In principle this should allow us to describe the $L$-packets as well. Roughly speaking, what we would like is to view all the ingredients - the toral invariant, the $L$-embedding, and the character $\psi$ that was used to construct the epipelagic representation - inside a common group, in fact the group $C_{\lambda}^{\vee}$ which is Pontrjagin dual to $C_{\lambda}$, and to cut out the L-packets as conditions on their position within $C_{\lambda}^{\vee}$. 

First let's recall the broad picture. An irreducible epipelagic representation is attached to a stable functional $\lambda$ and a character $\psi$ of $C_{\lambda} := \prod_i  \mu_2$. Under the local Langlands correspondence, we attach to $(\lambda, \psi)$ a character of $S$, hence a Langlands parameter $\varphi_{\chi} \co W_F \rightarrow {}^L S$. The parameter is determined on wild inertia by \eqref{character on wild inertia}. Note that the expression in \eqref{character on wild inertia} only depends on the stable orbit of $\lambda$, and not on $\psi$. 

The character of $S$ has the form $\chi = \chi_{\lambda} \cdot \psi \cdot \epsilon_\lambda$, where $\epsilon_\lambda $ and $\psi$ factor through $S/S_{1/2} \cong C_{\lambda}$. Thus 
\[
\varphi_{\chi} = \varphi_{\chi_{\lambda}} \varphi_{\psi} \varphi_{\epsilon_\lambda}.
\]

\begin{lemma}\label{tame cohomology equals character group}
The group is $C_{\lambda}^{\vee}$ is isomorphic to the subgroup of tamely ramified classes in $H^1(W_F, \wh{S})$.
\end{lemma}

\begin{proof}
This follows from class field theory for $\wh{S}$, and was already proved in the beginning of the discussion of \S \ref{section compute admissible embedding}.
\end{proof}

Using the lemma and the embedding
\[
H^1(W_F, \wh{S}) = \bigoplus_i H^1(W_{F_i}, \wh{U}_1(E_i/F_i)) \hookrightarrow \bigoplus_i H^1(W_{E_i}, \CC^{\times})
\]
we can view characters on $S/S_{1/2}$, inflated characters of $S$, as cohomology classes in $\bigoplus_i H^1(W_{E_i}, \CC^{\times})$. Moreover, the condition of triviality on $S_{1/2}$ forces its image cohomology class to be unramified. Indeed, a character of $S/S_{1/2}$ is a character of $\prod_i U_1(E_i/F_i)$ that vanishes on elements which are $1 \mod \varpi_E$ in each component, hence pull back to unramified characters of $E_i^{\times}$ via the map $E_i^{\times} \xrightarrow{x \mapsto x \ol{x}^{-1}} U_1(E_i/F_i)$.

Similarly, the $L$-embedding was determined by a tamely ramified class in $H^1(W_F, \wh{S}) $, under which group the admissible splittings $w \mapsto r_\lambda(w)$ formed a torsor, and this tamely ramified class again restricts to an unramified class in $\bigoplus_i H^1(W_{E_i}, \CC^{\times})$. The Langlands parameter attached to $\rho_{\lambda, \psi}$ is then explicitly given by 
\begin{align*}
W_F &  \rightarrow {}^L \wh{\mrm{U}}_n \\
w & \mapsto \varphi_{\psi}(w) \varphi_{\epsilon_\lambda}(w) r_\lambda(w) n(\omega_\lambda(w)) \times w .
\end{align*}

We want to know when $\rho_{\lambda, \psi}$ and $\rho_{\lambda', \psi'}$ have the same Langlands parameter. Although local Langlands parameters are considered modulo conjugacy, by demanding that wild inertial map in a fixed way into a fixed maximal torus $\wh{T} \subset \wh{G}$, with image having centralizer $\wh{T}$ by definition \ref{epipelagic parameter}, we have rigidified the parameters up to $\wh{T}$-conjugacy. Therefore, $\rho_{\lambda, \psi}$ and $\rho_{\lambda', \psi'}$ have equivalent Langlands parameters if and only if
 \[
\varphi_{\psi}(w) \varphi_{\epsilon_\lambda}(w) r_\lambda(w) n(\omega_\lambda(w)) \quad \text{ is $\wh{T}$-conjugate to } \quad \varphi_{\psi'}(w) \varphi_{\lambda'}(w) r_{\lambda'}(w) n(\omega_{\lambda'}(w)).
\]

To digest this condition, we will translate all of the data above back to $C_{\lambda}^{\vee}$. For $\epsilon_\lambda, \psi, \epsilon_\lambda', \psi'$, it is obvious how to view them as characters on $C_{\lambda} = S/S_{1/2}$, i.e. as elements of $C_{\lambda}^{\vee}$, and we denote their images by $[\epsilon_{\lambda}], [\psi]$ etc. to contrast with their appearance above as elements of $\bigoplus_i H^1(W_{E_i}, \CC^{\times})$.

 We can view the difference between two admissible embeddings, given by $w \mapsto r_{\lambda}(w) r_{\lambda'}(w)^{-1}$, as defining a tamely ramified cohomology class $[r_{\lambda}-r_{\lambda'}] \in \bigoplus_i H^1(W_{F}, \wh{S}_i)$, which by Lemma \ref{tame cohomology equals character group} can be identified with an element of $C_{\lambda}^{\vee}$. (What is being used here is that for stably conjugate $\lambda$, the associated tamely ramified tori are abstractly isomorphic, i.e. we have a canonical isomorphism $X^*(S_i) \cong X^*(S')$ as Galois modules, since these are determined by the partition $d_1+ \ldots + d_m =  n$.) The $\wh{T}$-conjugacy changes the cocycle $w \mapsto  r_\lambda(w) n(\omega_\lambda(w)) \times w$ by a coboundary. Thus, the $\wh{T}$-conjugation ambiguity is entirely encoded by  $[r_{\lambda}-r_{\lambda'}] \in  C_{\lambda}^{\vee}$. This discussion proves: 
 
 \begin{lemma} \label{lemmamostgeneral}
 The representations $\rho_{\lambda, \psi}$ and $\rho_{\lambda', \psi'}$ have the same Langlands parameter if and only if 
 \begin{equation}\label{equal Langlands parameters}
[\epsilon_\lambda] - [\epsilon_{\lambda'} ]  + [r_{\lambda}-r_{\lambda'}] = [\psi'] - [\psi] \in C_{\lambda}^{\vee}.
\end{equation}
\end{lemma}

We now substitute the expressions for the toral invariant and admissible embedding that we have computed. By Corollary \ref{cor: toral invariant 2}, we have
\begin{equation}\label{diff epsilons}
[\epsilon_\lambda-\epsilon_{\lambda'}]  (c_i)=  \leg{\eta_i}{q^{d_i}}^{n} / \leg{\eta_i'}{q^{d_i}}^{n}.
\end{equation}
Note that this is trivial if $n$ is even, and is $\leg{\eta_i}{q^{d_i}} / \leg{\eta_i'}{q^{d_i}}$ if $n$ is odd. 

We view $[r_{\lambda}-r_{\lambda'}]$ as a character on $C_{\lambda}^{\vee}$,  by sending $c_i$ \eqref{c_i} to the ratio of the expressions \eqref{r_p computed} for the two admissible embeddings:
\begin{equation}\label{diff rps}
[r_\lambda-r_{\lambda'}](c_i) = \prod_{i \neq j}   \prod_{t=1}^{(d_i,d_j)} \leg{\delta_i - \delta_j^{q^{t-1}}}{q^{[d_i,d_j]}} /  \leg{\delta_i' - (\delta_j')^{q^{t-1}}}
{q^{[d_i,d_j]}}.
\end{equation}

Plugging these equations into \eqref{equal Langlands parameters}, we obtain something which is ``concrete'' enough but quite a mess, since both $[\epsilon_\lambda]$ and $[r_{\lambda}]$ were described by extremely complicated formulas. We next proceed, in \S \ref{Orbit parametrization}, to give a somewhat cleaner characterization by relating $[\epsilon_\lambda] + [r_\lambda]$ to the position of the orbit of $\lambda$ within its stable orbit.

\subsection{Parametrization of orbits and L-packets	}\label{Orbit parametrization}

As discussed in \S \ref{moy prasad filtration}, the action of $\msf{G}_x$ on $\msf{V}_{x}$ is 
$\SO_n(k)$ acting on $\Sym^2(\mrm{Std})$. The \emph{stable orbit} of $\lambda \in  \msf{V}_{x}$ is defined to be the intersection of $\msf{V}_{x}(k)$ with the orbit of $\lambda$ under $\SO_{n}(\ol{k})$ in $\msf{V}_{x} \otimes \ol{k}$. Since $D_{\lambda} :=\prod_i \Res_{k_i/k} \mu_2 $ is the stabilizer of $\lambda$, the $k$-rational orbits of $\lambda$ within the stable orbit are a torsor for $\ker \left( H^1(k, D_{\lambda} ) \rightarrow H^1(k, \SO_n)\right)$, which is all of $\ker H^1(k, D_{\lambda})$ by Lang's theorem. 

We will explain a way to choose a basepoint for this torsor, which comes from a ``Kostant-Weierstrass section''. Using this, we can identify the position of the rational orbit of $\lambda$ in its stable orbit with an element of  $H^1(k, D_{\lambda} )$. There is a perfect pairing 
\[
\prod_i \Res_{k_i/k} \mu_2  \times \prod_i \Res_{k_i/k} \mu_2  \rightarrow \mu_2
\]
inducing (by Tate duality for finite fields) 
\[
H^1(k, \prod_i \Res_{k_i/k} \mu_2 ) \cong H^0(k, \prod_i \Res_{k_i/k} \mu_2 )^{\vee} \cong C_{\lambda}^{\vee}.
\]
Thus, the choice of a basepoint allows us to parametrize the rational orbit of $\lambda$ within its stable orbit by an element of $C_{\lambda}^{\vee}$. \\

Clearly, a basepoint for each stable orbit can be described by giving a section of 
\[
\msf{V}_x // \SO_n  \rightarrow \msf{V}_x.
\]
According to \cite[Theorem 28]{RLYG} such a section always exists for the representations under consideration, since they arise from the Vinberg-Levy theory of graded Lie algebras by \cite{RY14} Theorem 4.1\footnote{Paul Levy has informed us that in our setting the main idea for the existence of a section is already contained in early work of Kostant and Rallis \cite{KR71}, and that the relevant case of \cite[Theorem 28]{RLYG} is really due to Panyushev.}. We will pick a particular such section, and call it a \emph{Kostant-Weierstrass section.}\footnote{In \cite{RLYG} any such section is simply called a \emph{Kostant section}, but this may cause confusion with the special sections with this name in classical invariant theory.} Consider the algebraic group $\GL_n$ over $k$, and suppose we have an involution of $\GL_n$ with fixed subgroup $\mrm{O}_{n}$. This induces a decomposition $\mf{gl}_{n} = \mf{o}_{n} \oplus \mf{gl}(1)$ where $\mf{gl}(1)$ is the space of self-adjoint matrices. (In our case, $\mf{gl}(1) \cong \msf{V}_x$. The notation here follows that of the Vinberg-Levy theory in \cite{RY14}.) The quotient $\mf{gl}(1)//\SO_n$ is regular, and in fact is the affine space parametrizing characteristic polynomials. In this case we can write down an explicit section $\mf{gl}(1)//\SO_{n} \rightarrow \mf{gl}(1)$, in the form of a subspace $\mf{c} \subset \mf{gl}(1)$ which projects isomorphically down to $\mf{gl}(1)//\SO_{n}$:
\begin{equation}\label{Kostant section}
\mf{c} := \left\{\begin{pmatrix} 
a_1  & \ldots & a_{n-2} &  a_{n-1} & a_{n} \\ 
1   & & & &  a_{n-1} \\
 &1   & & &  a_{n-2}  \\
& & \ddots  & & \vdots \\
&  & & 1 & a_1
\end{pmatrix}  \right\}
\end{equation}
Here all the inner entries are $0$, to make the construction work well in all characteristics $>2$.

\begin{prop}\label{kostant identity} Consider an $n$-dimensional quadratic space $V$ over $k$ of characteristic  $>2$, with the bilinear form
\[
\langle (x_1, \ldots, x_n), (y_1, \ldots, y_n) \rangle = x_1y_n + x_2y_{n-1} + \ldots + x_n y_1.
\]
Let $\lambda$ be a self-adjoint operator on $V$, with characteristic polynomial $p(T)$. Suppose $p(T) =\prod_i p_i(T)$ with $p_i$ irreducible, and let $k_i = k[T]/p_i(T)$. We may write $\langle x, y \rangle = \sum \Tr_{k_i/k}(\eta_i xy )$, for some $\eta_i \in k_i$. Let $\delta_i$ be a root of $p_i(T)$. Then the function 
\begin{equation}\label{eq: rel pos cocycle}
c_i \mapsto  (-1)^{d_i-1} \leg{(-1)^{\lfloor d_i/2\rfloor}}{q}  \leg{\eta_i}{q^{d_i}} \prod_{j \neq i}\prod_{t=1}^{(d_i,d_j)} \leg{\delta_i - \delta_j^{q^{t-1}}}{q^{[d_i,d_j]}} 
\end{equation}
viewed as an element of $C_{\lambda}^{\vee} $ is exactly the position of $\lambda$ relative to the Konstant section \eqref{Kostant section}.
\end{prop}

\begin{proof}
By picking a vector $v \in V$ which is cyclic, i.e. such that $\{T^i v\}$ spans $V$, we may identify $V \cong k[x]/p(x)$. Then any element of $V$ can be represented (uniquely) by a polynomial 
\[
v(x)= v_{n-1} x^{n-1} + \ldots + v_0.
\] 
By Lemma \ref{disc weird pairing}, the pairing $\langle \cdot, \cdot \rangle$ on $V$ can be written as 
\[
\langle u,v \rangle = \omega_{n-1}(\alpha uv)
\]
for a unique $\alpha \in k[x]/p(x)$, where $\omega_{n-1}(u)$ is the coefficient of $x^{n-1}$ in the unique expression for $u$ as a polynomial of degree at most $n-1$.

Let $h_i(x) = p(x)/p_i(x)$. The decomposition $V  \cong \bigoplus  V_i$ can be realized with $V_i = \{ v(x) \colon v \leq d_i-1\}$ via the map $V_i \rightarrow V$ given by $v(x) \mapsto v(x) h_i(x)$. The restriction $\langle \cdot, \cdot \rangle|_{V_i}$ is then given by 
\[
\langle u,v \rangle_{V_i}  = \langle uh_i, vh_i \rangle = \omega_{n-1}(\alpha uvh_i^2).
\]
We aim to rewrite this in terms of the pairing of Lemma \ref{disc weird pairing} for $V_i \cong k[x]/p_i(x)$. If $\ol{\alpha uvh_i}$ is the representative for $\alpha uv h_i$ of deg $\leq d_i-1$ under the ``reduction mod $p_i$'' map $k[x]/p(x) \rightarrow k[x]/p_i(x)$, then (since $h_i$ is monic of degree $n-d_i$) we have
\[
\omega_{n-1}(\alpha uvh_i^2) = \omega_{d_i-1}(\ol{\alpha uv h_i}).
\]
By Lemma \ref{disc weird pairing} if $D_i$ denotes the discriminant of $\langle \cdot, \cdot \rangle|_{V_i}$ then we have $\leg{D_i}{q} = \leg{(-1)^{\lfloor d_i/2 \rfloor}}{q} \leg{\ol{\alpha h_i}}{q^{d_i}}$. Combining this with Lemma \ref{disc trace pairing}, we then have
\begin{equation}\label{eq: discriminants 0}
(-1)^{d_i-1} \leg{\eta_i}{q^{d_i}} =   \leg{D_i}{q} =  \leg{ (-1)^{\lfloor d_i/2 \rfloor}}{q} \leg{\ol{\alpha h_i}}{q^{d_i}}.
\end{equation}

We will use this equation to re-express the right hand side expression above with the right hand side of \eqref{eq: rel pos cocycle}. By Lemma \ref{Legendre norm} and the identity $\prod_{t=1}^{(d_i,d_j)} \Nm_{k_{[d_i,d_j]}/k_{d_i}}(\delta_i - \delta_j^{q^{t-1}})  = \prod_{t=1}^{d_j} (\delta_i - \delta_j^{q^{t-1}})$, we have 
\begin{equation}\label{eq: discriminants 1}
\prod_{t=1}^{(d_i,d_j)} \leg{\delta_i - \delta_j^{q^{t-1}}}{q^{[d_i,d_j]}} = \prod_{t=1}^{(d_i,d_j)}\leg{\Nm_{k_{[d_i,d_j]}/k_{d_i}}(\delta_i - \delta_j^{q^{t-1}})}{q^{d_i}}  =  \prod_{t=1}^{d_j} \leg{\delta_i - \delta_j^{q^{t-1}}}{q^{d_i}}
\end{equation}
On the other hand, from the definition of $h_i$ we compute directly that 
\begin{equation}\label{eq: discriminants 2}
\leg{\ol{\alpha h_i}}{q^{d_i}} = \leg{\ol{\alpha}}{q^{d_i}}  \prod_{j \neq i} \leg{p_j(\delta_i)}{q^{d_i}} = \leg{\ol{\alpha}}{q^{d_i}} \prod_{j \neq i} \prod_{t=1}^{d_j} \leg{\delta_i-\delta_j^{q^{t-1}}}{q^{d_i}}. 
\end{equation}

Substituting \eqref{eq: discriminants 0}, \eqref{eq: discriminants 1}, and \eqref{eq: discriminants 2} into the right hand side of \eqref{eq: rel pos cocycle}, it simplifies to $c_i \mapsto \leg{\ol{\alpha}}{q^{d_i}}$. We want to show that this cocycle represents the cohomology class measuring the relative position of $\lambda$ with respect to the Kostant section. For this it suffices to show that each member of the Konstant section has $\alpha=1$. To prove this, let 
$\lambda$ be a member of the Kostant section, written in terms of a basis $e_1, \ldots, e_n$ as 
\[
\lambda = \begin{pmatrix} 
 a_1  & \ldots & a_{n-2} &  a_{n-1} & a_{n} \\ 
1   & & & &  a_{n-1} \\
 &1   & & &  a_{n-2}  \\
& & \ddots  & & \vdots \\
&  & & 1 & a_1
\end{pmatrix} .
\]
Then it is easy to check that $e_1$ is a cyclic vector for $\lambda$ acting on $V$, so that we have an identification $V \cong k[x]/p(x)$ via $x^i = \lambda^i e_1$. We easily compute that $x^i = e_i + \text{(lower index terms)}$, so that
\[
\langle 1, x^i \rangle  = \delta_{i,n-1}.
\]
 By definition, $\alpha$ is such that 
\[
 \langle 1, x^i \rangle = \omega_{n-1}(\alpha x^i),
\]
so that $\omega_{n-1}(\alpha x^i) =  \delta_{i,n-1}$. This identities are  satisfied by $\alpha=1$, so by non-degeneracy $\alpha=1$ is the unique solution. 
\end{proof}

Let $\lambda_{\mrm{KW}}$ be the Kostant-Weierstrass section of $\lambda$ corresponding to \eqref{Kostant section}. As explained at the beginning of this subsection, we may view $[\lambda-\lambda_{\mrm{KW}}] \in C_{\lambda}^{\vee}$. For another $\lambda'$ in the stable orbit of $\lambda$, we have $[\lambda-\lambda'] = [\lambda- \lambda_{\mrm{KW}}] - [\lambda' - \lambda_{\mrm{KW}}] \in C_{\lambda}^{\vee}$.
 
\emph{We now restrict our attention to $n$ odd.} Combining Proposition \ref{kostant identity} with \eqref{diff epsilons} and \eqref{diff rps}, we see that in the notation of \eqref{equal Langlands parameters} we have
\[
[\epsilon_{\lambda}-\epsilon_{\lambda'}] + [r_{\lambda}-r_{\lambda'}]    = [\lambda - \lambda'] \in C_{\lambda}^{\vee}.
\]
(We have used here that the extra signs $(-1)^{d_i-1} \left( \leg{(-1)^{\lfloor d_i/2 \rfloor} }{q} \right)$ from Proposition \ref{kostant identity} cancel out when taking the ratio of the cocycles corresponding to two stably conjugate functionals.) Feeding this into \eqref{equal Langlands parameters}, we conclude: 

\begin{thm}\label{unitary epi param}
Consider $U_n$ with $n$ odd. The epipelagic representation $\rho_{\lambda, \psi}$ and $\rho_{\lambda', \psi'}$ lie in the same $L$-packet if and only if $[\lambda-\lambda']  = [\psi']- [\psi] \in C_{\lambda}^{\vee}$.
\end{thm}

For special unitary epipelagic representations, the epipelagic representations coming from the point $x$ are just restrictions from $\mrm{U}_n$, according to Lemma \ref{SU_n restriction}. Let us point out how the discussion changes for $\mrm{SU}_n$. 
\begin{itemize}
\item The centralizer $C_{\lambda} \cap {\SU_n}$ is cut out in $C_{\lambda}$ by the equation $\det = 1$. Therefore it is a subgroup of index $2$ unless all $d_i$ are even (which of course cannot happen if $\sum d_i = n$ is odd), and in the latter case it is all of $C_{\lambda}$.  
\item The representations $\rho_{\lambda, \psi}$ and $\rho_{\lambda', \psi'}$ of $\mrm{U}_n$ collapse if and only if $\lambda$ is rationally conjugate to $\lambda'$, and $\psi|_{C_{\lambda} \cap {\SU_n}} = \psi'|_{C_{\lambda} \cap {\SU_n}}$ (by Lemma \ref{lem: isom iff rat conj}).
\item The Langlands parameter for $\rho_{\lambda, \psi}|_{\SU_n}$ is then just the quotient of the Langlands parameter for $\rho_{\lambda, \psi}$ by the center of ${}^L \mrm{U}_n$. This exactly collapses two Langlands parameters which differ by the diagonal matrix $\mrm{diag}(-1, \ldots, -1)$.
\end{itemize}
Let $z \in C_{\lambda}^{\vee}$ be the character defined by $z(c_i) = -1$ for all $i$, which corresponds to the aforementioned diagonal matrix.

\begin{cor}\label{special epi param} Consider $\SU_n$ with $n$ odd. The epipelagic representation $\rho_{\lambda, \psi}|_{\SU_n}$ and $\rho_{\lambda, \psi}|_{\SU_n}$ lie in the same $L$-packet if and only if 
\[
[\lambda-\lambda'] =  ([\psi']-[\psi])  \in (C_{\lambda} \cap \SU_n)^{\vee} \quad \text{or} \quad 
[\lambda-\lambda'] =  z + ([\psi']-[\psi])  \in (C_{\lambda} \cap \SU_n)^{\vee} .
\] 
\end{cor}

\appendix

\section{Some results on discriminants}\label{app: disc}
Here we collect some facts about the Legendre symbols of discriminants over finite fields.  These results may be ``well known'' (Lemma \ref{disc trace pairing} especially), but we did not find a reference, and have opted to provide the proofs ourselves.

We recall the setup. Let $k$ be a finite field of size $q$ and $k_d$ its unique extension of degree $d$ for each $d \ge 1$. If $(V, \langle, \rangle )$ is a quadratic space over $k$, we define its discriminant relative to a $k$-basis $\left\{ v_1, \ldots, v_n \right\}$ of $V$ to be $\det \left( \langle v_i, v_j \rangle \right)$. Choosing a different basis changes the discriminant by a square in $k$. Therefore, if we denote by $\left( \frac{\cdot}{q} \right)$ the Legendre symbol on $k$, then $\leg{\disc(V, \langle, \rangle )}{q}$ is well-defined (i.e. independent of a choice of basis).

Now consider $k_d$ as an $k$-vector space. Any non-degenerate symmetric bilinear form on $k_d$ for which multiplication by $\lambda \in k_d$ is self-adjoint can be realized as 
\[
 \langle x, y \rangle_{\alpha} = \Tr_{ k_d / k} (\alpha xy) \qquad \textnormal{ for some } \alpha \in k^{\times}
 \]
 by the non-degeneracy of the trace pairing.

\begin{lemma}\label{Legendre norm}
We have $ \left( \frac{\Nm_{k_d / k}(\alpha)}{q} \right) = \left( \frac{\alpha}{q^d} \right)$
for $\alpha \in k_d$. In particular, for $\alpha \in k$ we have $\left( \frac{ \alpha^d }{q} \right) = \left( \frac{\alpha}{q^d} \right)$.
\end{lemma}

\begin{proof}
The norm map is a surjective group homomorphism, and clearly preserves the property of being a square. Since there are as many squares as non-squares in each of $k^{\times}$ and $k_d^{\times}$, it must be the case that non-squares in $k_d^{\times}$ are mapped by $\Nm_{k_d/k}$ to non-squares in $k^{\times}$. 
\end{proof}

\begin{lemma}\label{disc trace pairing} Let $D_{\alpha}$ be the discriminant of $\left(  k_d, \langle , \rangle_{\alpha} \right)$. Then we have
\[ \left( \frac{D_{\alpha}}{q} \right) = (-1)^{d-1} \cdot \left( \frac{\alpha}{q^d} \right). \]
\end{lemma}

\begin{proof}
We start by showing how the formula for $D_{\alpha}$ for general $\alpha$ follows from that of $D_1$. Notice that on the right hand side of the formula we aim to prove, the only dependence on $\alpha$ is on the third factor, which is clearly $1$ when $\alpha =1$. Suppose then that we proved that 
\begin{equation} \label{delta1} \left( \frac{D_1}{q} \right) = (-1)^{d-1}.
\end{equation}
To prove the formula for $D_{\alpha}$, it then suffices to show that $\left( \frac{D_{\alpha}}{q} \right) = \left( \frac{D_1}{q} \right)  \cdot \left( \frac{\alpha}{q^d} \right)$. 

Fix a primitive element $x$ for $k_d \supset k$, and consider the basis $\{ x^i \}_{i=1}^d$. We want to show that \[ \left( \frac{\det \left( \Tr_{k_d / k} (\alpha x^i x^j ) \right) }{q} \right) = \left( \frac{\det \left( \Tr_{k_d / k} ( x^i x^j ) \right) }{q} \right) \cdot \left( \frac{\alpha}{q^d} \right). \]
Notice that \[ \Tr (\alpha x^i x^j) = \sum_{k=1}^d (\alpha x^i x^j)^{q^k} = \sum_{k=1}^d (\alpha x^i)^{q^k} \cdot (x^j)^{q^k} \]
so in terms of the two $d \times d$ matrices $\left( A_{\alpha} \right)_{i,k} = \alpha^{q^k} (x^i)^{q^k}$ and $B_{k,j} = (x^j)^{q^k}$ we have $\Tr (\alpha x^i x^j) = (A_{\alpha} \cdot B)_{i,j}$.
In particular, $\det \left( \Tr (\alpha x^i x^j) \right) = \det (A_{\alpha}) \cdot \det B$. 

Now, factoring out $\alpha^{q^k}$ from the $k$th column of $A_{\alpha}$  makes it clear that
 \[ \det (A_{\alpha}) = \prod_{k=1}^d \alpha^{q^k} \cdot \det (A_1) = \Nm_{k_d / k}(\alpha) \cdot \det (A_1). \]
In particular, $ \det \left( \Tr(x^i x^j) \right) = \det (A_1) \cdot \det B$, so by Lemma \ref{Legendre norm} the difference between $\leg{D_{\alpha}}{q}$ and $\leg{D_1}{q}$ is $\Nm_{k_d / k}(\alpha)  = \leg{\alpha}{q^d}$, as desired. \\

It remain to prove \eqref{delta1}. Choosing the basis $\{ x^{i} \}$ as before, saw that 
\[
\det \left( \Tr_{k_d / k} (v_i v_j) \right) = \det (A_{ik}) \det (B_{kj}) = \left( \det (A_{ij}) \right)^2
\]
where $A_{ik} = (x^i)^{q^k} = x^{iq^k} = \left( x^{q^k} \right)^i$, and $\det (B_{ij}) = \det (A_{ik})$ because $B$ is simply the transpose of $A$. Using that $(A_{ik})$ is a Vandermonde matrix in the variables $\{x^q, \ldots, x^{q^d} \}$, we find that
 \[ \det \left( \Tr_{k_d / k} (v_i v_j) \right) = \left( \prod_{1 \le i < j \le d} \left( x^{q^j} - x^{q^i} \right) \right)^2. \]
We need to understand when this quantity is a square in $\F_q$. It is obviously a square in $\F_{q^d}$, so it is a square in $\F_q$ if and only if $\prod_{1 \le i < j \le d} \left( x^{q^j} - x^{q^i} \right)$ is already in $\F_q$. This is the case if and only if it is fixed by $\Frob_q$. Applying $\Frob_q$ permutes the factors, but with $d-1$ sign changes coming from the terms indexed by $(i,j=d)$, hence changes the product by the sign $(-1)^{d-1}$, which is exactly what we wanted to find. 
\end{proof}

Let us extract a non-obvious consequence of this result. 

\begin{cor}\label{disc multiplier}
Let $B \co k_d  \rightarrow k$ be a symmetric bilinear pairing for which multiplication by $\lambda \in k_d$ is self-adjoint. Then $B_{\alpha}(x,y) := B(\alpha x, y)$ is another pairing with the same property, and 
\[
\leg{B_{\alpha}}{q} = \leg{B}{q} \leg{\alpha}{q}.
\]
\end{cor}

\begin{proof}
We have $B(x,y) = \Tr_{k_d/k}(\beta xy)$ for some $\beta$, and then the result follows from Lemma \ref{disc trace pairing}. 
\end{proof}

We next study the discriminant of a pairing with a different form. Fix a primitive element $x \in k_d$ so that $k_d \cong k[x] / p(x)$. Then every element $a \in k_d$ admits a unique representation  
\[
a = \sum_{i=0}^{d-1} a_i x^i, \quad a_i \in k.	
\]
We define the $k$-linear functional $\omega: k_d \longrightarrow k$ by $\omega(a) = a_{d-1}$. We set 
 \[ 
 \left( a,b \right)_{\alpha} = \omega(\alpha ab) \qquad \forall \alpha \in k_d^*.
  \]
For every $\alpha$, this is a symmetric bilinear pairing $k_d \times k_d \rightarrow k$ for which multiplication by any $\lambda \in k_d$ is self-adjoint. 

\begin{lemma}\label{disc weird pairing} 
Every symmetric bilinear pairing $k_d \times k_d \rightarrow k$ for which multiplication by any $\lambda \in k_d$ is self-adjoint agrees with $(\cdot, \cdot)_{\alpha}$ for some $\alpha$.  Moreover, let $D_{\alpha}'$ be the discriminant of $(\cdot, \cdot)_{\alpha}$. Then 
\[ 
\left( \frac{D_{\alpha}'}{q} \right) = \leg{(-1)^{\lfloor d/2\rfloor}}{q}  \left( \frac{\alpha}{q^d} \right). 
\]
\end{lemma}

\begin{proof}
The first statement follows from the non-degeneracy of the form $(\cdot, \cdot)_1$ and counting. Thanks to Corollary \ref{disc multiplier}, it suffices to prove the second claim for $\alpha=1$. Taking as basis $\{ x^i \}_{i=0}^{d-1}$, we need to compute $\det \left( (x^i, x^j)_1 \right)_{i,j=0 \ldots d-1} = \det \left( \omega(x^{i+j}) \right)_{i=0, \ldots, d-1}$, which we will do by hand. 

In particular, we notice that the first row of the corresponding matrix $A_1$, where $i=0$, has the only nonzero element in the last column since $\omega(x^0 x^j) \neq 0 \iff j=d-1$, in which case $\omega(x^{d-1})=1$.

We can consider the Laplace expansion along the first row then. In the second row ($i=1$) we have that $\omega(x^1 x^j) \neq 0 \iff j=d-2, d-1$. Since we do not care about the last column, the only contribution to our determinant when we consider the Laplace expansion along the second row is that of $\omega(x^1 x^{d-2}) =1$.

Continuing in this way, we obtain that 
\[ 
\det A = (-1)^{d+1} \cdot (-1)^{d-1+1} \cdot \ldots \cdot (-1)^{1+1} = \prod_{i=1}^{d} (-1)^{i+1} =  (-1)^{\sum_{j=0}^{d-1} j} = (-1)^{\frac{d(d-1)}{2}}. 
\]
Now a case-by-case verification shows that for each possible residue class of $d \bmod 4$, we have that $ (-1)^{\frac{d(d-1)}{2}} = (-1)^{\floor{\frac{d}{2}}}$ as desired.

\end{proof}

\bibliographystyle{amsalpha}

\bibliography{Bibliography}

\end{document}